\documentclass[11pt]{article}
\usepackage{amsmath}
\usepackage{amsfonts}
\usepackage{amssymb}
\usepackage{epsfig}
\usepackage{fancybox}
\usepackage{bbm}
\usepackage{mathrsfs}
\usepackage{bbm}
\usepackage{bm}
\usepackage{dsfont}
\usepackage{epstopdf}
\usepackage{multirow}
 \usepackage{xcolor}
\usepackage{float}
\usepackage{graphicx}
\usepackage{subfigure}
\usepackage{algorithm}
\usepackage{makecell}
\usepackage{boldline}
\usepackage{diagbox}
\usepackage{threeparttable}
\usepackage{adjustbox}
\usepackage{booktabs}
\usepackage[noend]{algpseudocode}
\usepackage[title]{appendix}
\usepackage[colorlinks=true]{hyperref}
\hypersetup{urlcolor=black,citecolor=blue,linkcolor=blue}
\usepackage{tikz}
\usetikzlibrary{shapes.geometric, arrows}
\usepackage{xcolor} 

\renewcommand{\Box}{\framebox{\rule{0.3em}{0.0em}}}

\def\prob {{\rm Prob}}

\newcommand{\st}{\inmat{s.t.}}

\newtheorem{theorem}{Theorem}[section]
\newtheorem{theorem*}{Theorem}[subsubsection]

\newtheorem{proposition}{Proposition}[section]
\newtheorem{example}{Example}[section]

\newtheorem{definition}{Definition}[section]
\newtheorem{assumption}{Assumption}[section]

\newcommand{\setd}{{ d \kern -.15em l}}
\newcommand{\hatsetd}{ d \hat{\kern -.15em l }}
\newcommand{\dd}{\mathsf {d\kern -0.07em l}}

\newcommand{\bgeqn}{\begin{eqnarray}}
\newcommand{\edeqn}{\end{eqnarray}}
\newcommand{\bgeq}{\begin{eqnarray*}}
\newcommand{\edeq}{\end{eqnarray*}}

\newcommand{\R}{{\rm I\!R}}

\newcommand{\inmat}[1]{\mbox{\rm {#1}}}

\graphicspath{{./Figures/}}

\newcommand{\Z}{{\cal Z}}

\newcommand{\be}{\begin{equation}}
\newcommand{\ee}{\end{equation}}

\renewcommand{\Box}{\hfill \rule{2.3mm}{2.3mm}}

\numberwithin{equation}{section}

\renewcommand{\Box}{\framebox{\rule{0.3em}{0.0em}}}

\def\prob {{\rm Prob}}

\newcommand{\vt}{{\vartheta}}

\newcommand{\mb}{\mathbb}

\newcommand{\mr}{\mathscr}

\def\Prob{{\rm{Prob}}}

\title{Randomization of Spectral Risk Measure and Distributional Robustness\thanks{This work is supported by a CUHK direct grant and CUHK start-up grant,  National Natural Science Foundation of China (12171145) and Postgraduate Scientific Research Innovation Project of Hunan Province (CX20210609)}}
\author{
Manlan Li\footnote{School of Mathematics and Computational Science, Xiangtan University, Xiangtan, China; Department of Systems Engineering and Engineering Management, The Chinese University of Hong Kong. E-mail: mlli@se.cuhk.edu.hk, liml@smail.xtu.edu.cn.
The work of this author is carried out when she 
works in the CUHK as a research assistant.
},
Xiaojiao Tong\footnote{School of Mathematics and Statistics, Hunan First Normal University,
          Changsha, China.         E-mail: dysftxj@hnfnu.edu.cn.}
and Huifu Xu\footnote{Department of Systems Engineering and Engineering Management, The Chinese University of Hong Kong. Email: hfxu@se.cuhk.edu.hk.}}

\date{\today}

\begin{document}

\maketitle

\begin{abstract}
In this paper, we consider a situation where a decision maker's (DM's) risk preference can be described by a spectral risk measure (SRM) but 
there is not a single SRM which can be used to represent 
the DM's preferences consistently. 
Consequently we
propose to
randomize the {\color{black}SRM} by introducing a random parameter in the risk spectrum. 
The 
randomized SRM (RSRM) allows one to describe the DM's preferences at different states with different SRMs. When the distribution of the random parameter is known, i.e., the randomness of the DM's preference can be described by a probability distribution, we introduce a new risk measure which is the mean value of the {\color{black}RSRM}.
In the case when the distribution is unknown, we propose a distributionally robust formulation of {\color{black}RSRM}.
The RSRM paradigm provides a new framework for interpreting the well-known Kusuoka's representation of law invariant coherent risk measures  and addressing inconsistency issues arising from 
observation/measurement errors or erroneous responses in preference elicitation process. 
We discuss in detail computational schemes for solving optimization problems based on the 
{\color{black}RSRM} and the distributionally robust {\color{black}RSRM}.

\end{abstract}

\noindent
\textbf{Key words.} Preference inconsistency, randomized SRM, distributionally robust RSRM, Kusuoka's representation

\section{Introduction}
\label{sec:introd}
{\color{black}

The concept of law invariant
coherent risk measure (LICRM) has been widely used in risk management and finance since its introduction
by Artzner et al. \cite{ADEH99}. 
LICRM specifies four important 
 principals of a risk measure which are 
shared by many 
investors in finance but it does not provide 
an advice as to  which
particular LICRM  an individual investor may use.
In a seminal work,
Kusuoka \cite{Kus01} links LICRM to conditional value-at-risk (CVaR)
which is the average
of tail losses 
above a specified confidence level. Specifically he
demonstrates that
any LICRM $\rho$ can be represented as
\begin{equation}
\setlength{\abovedisplayskip}{2pt}
\setlength{\belowdisplayskip}{2pt}
\label{eq:kusuoka's-representation}
\rho(X) =\sup_{\mu\in \mathcal{M}} \int_0^1 \inmat{CVaR}_\alpha(X)\mu(d\alpha),
\end{equation}
where 
$X:\Omega\rightarrow\R$ is a random variable 
representing losses 
and $\mathcal{M}$ denotes a set of probability measures on $[0,1]$.
In practice, we may interpret
(\ref{eq:kusuoka's-representation})
from {\em random risk measure} perspective.
Consider a DM whose risk preference can be described by $\inmat{CVaR}_\alpha$
but there is not a unique $\alpha$
which fits to the DM's risk preference.
For example, in some cases, the DM's preference may be described
by $\inmat{CVaR}_{0.95}$
whereas in other cases
it may be represented by
$\inmat{CVaR}_{0.90}$
or $\inmat{CVaR}_{0.99}$.
This often happens when the environment of the 
decision making problem such as macro-economic circumstance  changes.
In this case, we may
randomize $\alpha$ to capture the DM's varying 
preference.
The probability distribution
of $\alpha$ 
reflects the likelihood that the DM's risk preference can be described by $\inmat{CVaR}_\alpha$,  
e.g., with probability $0.2$, $0.3$ and $0.5$ the DM's risk preferences can be described by 
$\inmat{CVaR}_{0.90}$, $\inmat{CVaR}_{0.95}$  and $\inmat{CVaR}_{0.99}$  respectively. 
From modeller's perspective, 
we may use the expected value of $\inmat{CVaR}_\alpha$ 
to measure the DM's average risk preference, that is, $\int_0^1 \inmat{CVaR}_\alpha(X)\mu(d\alpha)$, where 
$\mu$ is the probability distribution of $\alpha$.
This kind of thinking is in line with random utility theory (\cite{Fis98,Kop01}) where
a DM's preference
has to be represented by a random utility function rather than a deterministic von Neumann-Morgenstern's utility function.
In the absence of complete information on the distribution of
$\alpha$, we may
use partially available information to construct a set of distributions $\mathcal{M}$ and consider the worst
average risk preference, which gives rise to the right hand side (rhs) of (\ref{eq:kusuoka's-representation}).
Note that ${\cal M}$ in (\ref{eq:kusuoka's-representation})
is unknown, we call it Kusuoka's ambiguity set.


Another well-known representation of LICRM
is in terms of weighted value-at-risk (VaR), that is, 
a real-valued function  $\rho$ is a
law invariant coherent risk measure
LICRM
if and only if
\begin{equation}
\setlength{\abovedisplayskip}{2pt}
\setlength{\belowdisplayskip}{2pt}
 \rho(X) :=\sup_{\sigma\in \mathfrak{A}}\int_0^1  F^{\leftarrow}_X(t)\sigma(t)dt,
\label{eq:RSRM}
\end{equation}
where $F^{\leftarrow}_X(t)$ is the quantile function of $X$, 
$\sigma: [0,1)\to \R_+ $ is a non-decreasing and right-hand side continuous function with $\int_{0}^1\sigma(t)dt=1$ 
(known as risk spectrum) 
and
$\mathfrak{A}$ is a set of risk spectra,
see \cite{Sha13,PiS15} and Chapter 6 in \cite{SDR21}.
In this formulation, each $\sigma$ signifies a DM's risk preference,
where
$\int_0^1  F^{\leftarrow}_X(t)\sigma(t)dt$
is known as a spectral risk measure (\cite{Ace02}),
and $\rho(X)$ is  the worst-case spectral risk measure calculated from $\mathfrak{A}$.
Like formulation
(\ref{eq:kusuoka's-representation}),
$\mathfrak{A}$ is unknown but may be
elicited via DM's preferences over pairwise comparison 
lotteries.

The two robust formulations
give rise to the same LICRM,
despite the robust frameworks and the underlying ingredients are different. The former uses
$\inmat{CVaR}_\alpha(X)$ as a building block
with risk preference being
signified by confidence level $\alpha$, whereas the latter uses quantile function $F^{\leftarrow}_X(t)$
(also known as  value at risk (VaR)) as a building block with $\sigma$ representing a DM's risk preference.
Pichler and Shapiro \cite{PiS15} show  that there exists a linear mapping $\mb{T}:\mathscr{M}\rightarrow\mathfrak{S}$ that is defined by $\sigma(\tau)=(\mb{T}\mu)(\tau):=\int_{0}^\tau(1-\alpha)^{-1}\mu(d\alpha)$ such that $\mu=\mb{T}^{-1}\sigma$ given by $\mu(\alpha)=(\mb{T}^{-1}\sigma)(\alpha)=(1-\alpha)\sigma(\alpha)+\int_{0}^\alpha\sigma(\tau)d\tau$ for all $\alpha\in [0,1]$, which means $\mathcal{M}=\{\mu=\mb{T}^{-1}\sigma:\sigma\in\mathfrak{A}\}$.
Using the relationship, we have
$   \rho_\sigma(X):= \int_{0}^1F^\leftarrow_X(t)\sigma(t)dt=\int_{0}^1\inmat{CVaR}_\alpha(X)\mu(d\alpha)$,
and can interpret the rhs of (\ref{eq:RSRM}) as follows: a DM's risk preference can be described by $F^{\leftarrow}_X(t)$ but there is no
unique $t$ such that $F^{\leftarrow}_X(t)$ captures the DM's preferences in all cases. Consequently, we  use a randomized $t$ and its distribution $\mu=\mb{T}^{-1}\sigma$ to calculate the average
 VaR and the worst average VaR value in the absence of complete information of $\mu$.
 Of course, we can also regard $\rho_\sigma(\cdot)$
as a building block, that is, use it to represent the DM's risk preference.
In this case,
(\ref{eq:RSRM}) becomes a preference robust SRM model rather than a distributionally robust preference model, see \cite{wx20}.
 
 \thispagestyle{empty}
\tikzstyle{startstop} = [rectangle, rounded corners, minimum width = 1.5cm, minimum height=1cm,text centered, draw = black]
\tikzstyle{io} = [trapezium, trapezium left angle=70, trapezium right angle=110, minimum width=2cm, minimum height=1cm, text centered, draw=black]
\tikzstyle{process} = [rectangle, minimum width=3cm, minimum height=1cm, text centered, draw=black]
\tikzstyle{decision} = [diamond, aspect = 3, text centered, draw=black]
\tikzstyle{arrow} = [->,>=stealth]
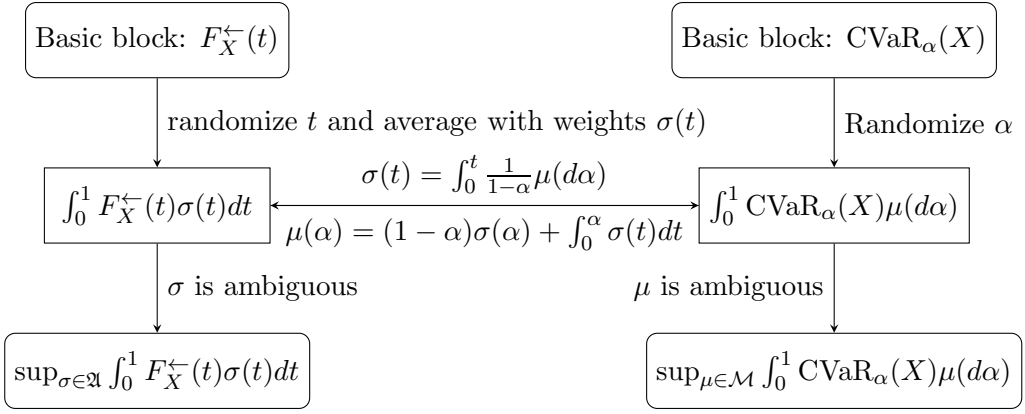
\begin{figure}
\centering
\begin{tikzpicture}[node distance=2cm]
\node[startstop](var)
{ Basic block:  $F_X^\leftarrow(t)$};
\node[startstop, right of = var, xshift = 7cm](cvar)
{ Basic block: $\inmat{CVaR}_\alpha(X)$};
\node[process, below of = var, yshift = -0.2cm](lcrm1)
{$\int_{0}^1F_X^\leftarrow(t)\sigma(t)dt$};
\node[process, below of = cvar, yshift = -0.2cm](lcrm2)
{$\int_{0}^1\inmat{CVaR}_\alpha(X)\mu(d\alpha)$};
\node[startstop, below of = lcrm1, yshift = -0.2cm](rlcrm1)
{$\sup_{\sigma\in\mathfrak{A}}\int_{0}^1F_X^\leftarrow(t)\sigma(t)dt$};
\node[startstop, below of = lcrm2, yshift = -0.2cm](rlcrm2)
{$\sup_{\mu\in\mathcal{M}}\int_{0}^1\inmat{CVaR}_\alpha(X)\mu(d\alpha)$};
\coordinate (point1) at (-3cm, -6cm);
\draw [arrow](var) -- node [right] {randomize $t$ and  average with weights $\sigma(t)$} (lcrm1);
\draw [arrow](cvar) -- node [right] {Randomize $\alpha$} (lcrm2);
\draw [arrow](lcrm1) -- node [right] {$\sigma$ is ambiguous} (rlcrm1);
\draw [arrow](lcrm2) -- node [left] {$\mu$ is ambiguous} (rlcrm2);
\draw [arrow](lcrm1) -- node [above] {$\sigma(t)=\int_{0}^t\frac{1}{1-\alpha}\mu(d\alpha)$ } (lcrm2);
\draw [arrow](lcrm2) -- node [below]
{$\mu(\alpha)=(1-\alpha)\sigma(\alpha)+\int_{0}^\alpha\sigma(t)dt$ } (lcrm1);
\end{tikzpicture}
\caption{Structural interpretations of Kusuoka's representation and 
spectral risk representation (\ref{eq:RSRM})
}
\label{eq:steam-licrm}
\end{figure}
\begin{figure}
\centering
\begin{tikzpicture}[node distance=2cm]
\node[startstop](srm)
{Basic block :  $\inmat{SRM}$ $\rho_\sigma(X)=\int_{0}^1F_X^\leftarrow(t)\sigma(t)dt$};
\node[process, below of = srm, yshift = -0.2cm](sdsrm)
{$\inmat{RSRM}$:  $\rho_{\sigma(\cdot,s)}(X)=\int_{0}^1F_X^\leftarrow(t)\sigma(t,s)dt$};
\node[process, below of = sdsrm, yshift = -0.2cm](asdsrm)
{$\inmat{ARSRM}$: $\rho_{Q}(X)=\mathbb{E}_Q\left[\int_{0}^1F_X^\leftarrow(t)\sigma(t,s)dt\right]$};
\node[startstop, below of = asdsrm, yshift = -0.2cm](rsdsrm)
{$\inmat{DR-ARSRM}$: $\rho_{\mathfrak{Q}}(X)=\sup_{Q\in\mathfrak{Q}}\mathbb{E}_Q\left[\int_{0}^1F_X^\leftarrow(t)\sigma(t,s)dt\right]$};
\coordinate (point1) at (-3cm, -6cm);
\draw [arrow](srm) -- node [left] {randomize $\sigma$} (sdsrm);
\draw [arrow](sdsrm) -- node [left] {$s\sim Q$} (asdsrm);
\draw [arrow](asdsrm) -- node [right] {$Q$ is ambiguous} (rsdsrm);
\end{tikzpicture}
\caption{Flow chart of the models in this paper}
\label{ep:steam-licrm}
\end{figure}
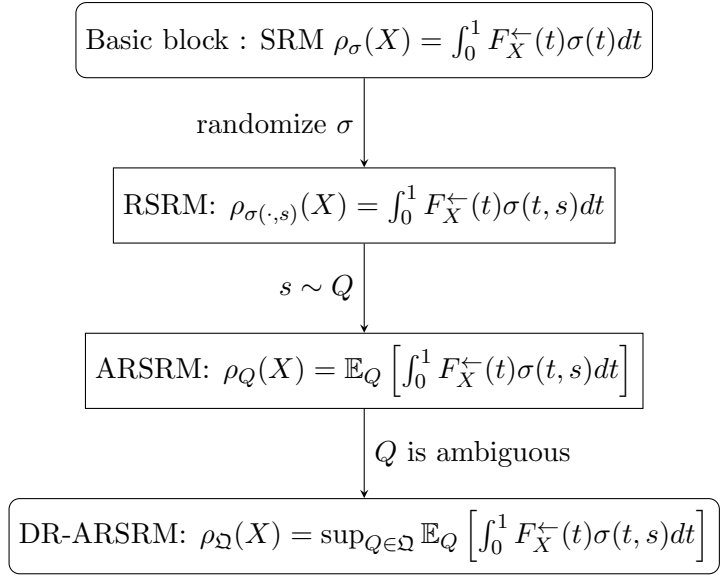
The discussions above show that there are at least three ways
to describe a DM's risk preference:
(a) Use CVaR or VaR but quite often one may find that there does not exist a unique CVaR or VaR 
because the frameworks are too small. 
(b) Randomize them via randomization of the confidence level $\alpha$ and then consider the mean of them under some probability distribution of $\alpha$ {\color{black} or $t$}.
In the CVaR case, 
it gives rise to $\int_{0}^1\inmat{CVaR}_\alpha(X)\mu(d\alpha)$
whereas in the VaR case, it 
leads to  $\int_{0}^1F_X^\leftarrow(t)\sigma(t)dt$.
Since the latter  can be equivalently written as $\int_{\R}x\sigma(F_X(x))dF_X(x))=
\int_\R x dg(F_x(x))$ for some $g$,
it is 
also  known as distortion risk measure (\cite{Den90,Wang95}) based on dual theory of choice (\cite{Yaa87})
since it is essentially about distortion of $F_X$.
(c) If the distribution of 
$\alpha$ {\color{black} or $t$} is ambiguous, then we may use a robust formulation for both.
 So which model should one select? The answer 
  depends on which model can be more conveniently 
used to 
  describe the DM's preference.
From modelling perspective, it might be more preferable
to use $\rho_\sigma(\cdot)$ and its robust formulation as this will bypass the inconsistency issue and the subsequent randomization step in modelling.
Moreover, since $\rho_\sigma(\cdot)$ is a coherent RM, it might cover a larger class of risk preference representation problems.

In the literature of risk management, there are
several versions of randomization of  risk measure.
For instance,
Zhu and Fukushima \cite{ZhF09} consider 
randomness of CVaR 
arising from uncertainty of the probability 
distribution of random loss $X$, denoted by $F$,
and propose a distributionally robust CVaR model
$\sup_{F\in {\cal F}}\inmat{CVaR}_\alpha^F(X)$
where ${\cal F}$ is an uncertainty set of $F$.
Qian, Wang and Wen \cite{QWW19} introduce a general framework of composite risk measure over the space
of randomized LICRMs. In both frameworks, 
the randomness arises from uncertainty/ambiguity of 
the probability distribution of random loss rather than the DM's risk preferences. 
{\color{black}Dentcheva and  Ruszczyński \cite{DR20} establish a new framework that addresses endogenous uncertainty in the context of risk-averse two-stage optimization models with partial information and decision-dependent observation distribution.}
Delage, Kuhn and Wiesemann \cite{DKW19} consider another type of random risk measure where the randomness comes from
randomization of the decision variables. That is, instead of considering random loss $X$, they consider $X_s$ where $s$ is a random parameter 
associated with randomization of decision making (such as flipping a coin).
Consequently 
the authors introduce a new risk measure which captures the uncertainty of both $X$ and $s$. Again, this kind of risk measure differs from ours because their randomization stems 
from random strategy rather than the DM's risk preference.


In this paper, we assume that a DM's risk preference can be represented by SRM $\rho_\sigma(\cdot)$ but there
exist
potential inconsistencies, which means that
there is no
unique $\sigma$ such that the DM's preference
can be represented by $\rho_\sigma(\cdot)$. 
In this case, we may use a parametric SRM
$\rho_{\sigma(\cdot,s)}(\cdot)$ to describe
the DM's varying risk preferences where
the risk spectrum $\sigma(\cdot,s)$ depends on 
parameter $s$. We call $s$ a state variable
representing the DM's ``state of risk preference''.
This should be differentiated from state variable 
of an action in Anscombe and Aumann's model 
(\cite{AnA63}). 
Since it is often uncertain
when the DM uses a particular parametric SRM, 
we may make $s$ as a random variable,
that is,
the DM's preference displays some kind of randomness,
and  
subsequently
consider the following randomized SRM:
\begin{equation*}
\setlength{\abovedisplayskip}{2pt}
\setlength{\belowdisplayskip}{2pt}
\label{eq:rdmSRM}
    \inmat{(RSRM)} \quad \rho_{\sigma(\cdot,s)}(X):= \int_{0}^1F^\leftarrow_X(t)\sigma(t,s)dt.
\end{equation*}
Let $Q$ denote the probability distribution of $s$. Then the average of $\rho_{\sigma(\cdot,s)}(X)$ can be calculated by
\begin{equation*}
\setlength{\abovedisplayskip}{2pt}
\setlength{\belowdisplayskip}{2pt}
   \inmat{(ARSRM)} \quad
   \rho_{Q}(X):= \mathbb{E}_Q\left[\int_{0}^1F^\leftarrow_X(t)\sigma(t,s)dt\right].
   \label{eq:Exp-RSRM} 
\end{equation*}
In the absence of complete information of $Q$, we may consider
a robust formulation
\begin{equation}
\setlength{\abovedisplayskip}{2pt}
\setlength{\belowdisplayskip}{2pt}
  {\rm (DR-ARSRM)} \quad
   \rho_{{\mathfrak{Q}}}(X):= \sup_{Q\in {\mathfrak{Q}}}
  \mathbb{E}_Q\left[\rho_{\sigma(\cdot,s)}(X)\right]
  = \sup_{Q\in {\mathfrak{Q}}}
  \mathbb{E}_Q\left[
  \int_{0}^1F^\leftarrow_X(t)\sigma(t,s)dt\right].
   \label{eq:Exp-RSRM-R}
\end{equation}
From the definition, we can see immediately that if $\sigma(\cdot,s)$ is a non-decreasing function for every $s$, then both $ \rho_{\sigma(\cdot,s)}$ and
$\rho_{{\mathfrak{Q}}}$ are coherent risk measures. The key issue is how to use
incomplete information about a DM to construct the ambiguity set ${\mathfrak{Q}}$ and how to solve problem (\ref{eq:Exp-RSRM-R}) efficiently.
A popular way in the literature of behavioural economics and preference robust optimization is
to use pairwise comparison questionnaires (see for example \cite{ArD15})  to elicit the DM's risk preference and 
subsequently construct ${\mathfrak{Q}}$. Another way is to assume that there is a nominal probability distribution $\hat{Q}$ based on empirical data or subjective judgement
and then construct ${\mathfrak{Q}}$ by a ball of
distributions centered at $\hat{Q}$. In this paper, we will discuss how the two approaches may be properly used.

The distributionally robust model (\ref{eq:Exp-RSRM-R}) may be applied to optimal decision making problems where $X$ is
interpreted as the loss of the problem such as
portfolio of financial investments. In this case, we replace $X$ by $f(z,\xi)$ where $z$ is the decision vector and $\xi$ is a vector of exogenous uncertainties. 
We consider the following minimax problem
\begin{equation}
\setlength{\abovedisplayskip}{2pt}
\setlength{\belowdisplayskip}{2pt}
   \min_{z\in {\cal Z}}\sup_{Q\in {\mathfrak{Q}}}
  \mathbb{E}_Q\left[\rho_{\sigma(\cdot,s)}(f(z,\xi))\right]
  = \min_{z\in {\cal Z}}
  \sup_{Q\in {\mathfrak{Q}}}
  \mathbb{E}_Q\left[
  \int_{0}^1F^\leftarrow_{f(z,\xi)}(t)\sigma(t,s)dt\right].
   \label{eq:Exp-RSRM-R-opti}
\end{equation}
Here we give a simple example to illustrate the model.
\begin{example}
An investor is considering an investment where the profit/loss
is closely related to 
future
uncertainties.
Let $X$ denote the overall loss. An analysis shows that $X$ is mainly
affected by
micro-economic circumstance such as
market demand and competitor's position.
However, the investor's risk attitude
is sensitive to the macro-economic situation,
that is,  he is more risk taking
when the macro-economic situation is good and risk averse otherwise.
Let $s$ denote
%
the state of the macro-economic situation
with two scenarios
$s_1$ and $s_2$.
Assume also that the investor's risk attitude can be represented by the SRM. Then we can  write down the average of
the investor's SRM of $X$ as
$    \rho_Q(X) = q_1\rho_{\sigma(\cdot,s_1)}(X)+q_2\rho_{\sigma(\cdot,s_2)}(X)
$, where $Q$ denotes the probability distribution of $s$, i.e., $Q(s=s_1)=q_1, Q(s=s_2)=q_2$.
This distribution reflects the investor's belief about uncertainty of macro-economic situation in future and is therefore subjective.
In the absence of complete information of $Q$, we may consider
a robust model
\begin{equation*}
\setlength{\abovedisplayskip}{2pt}
\setlength{\belowdisplayskip}{2pt}
\rho_\mathfrak{Q}(X) = \sup_{\boldsymbol{q}\in\mathfrak{Q}}\left\{q_1\rho_{\sigma(\cdot,s_1)}(X)+q_2\rho_{\sigma(\cdot,s_2)}(X)\right\},
\end{equation*}
which calculates the worst average from an ambiguity set of plausible distributions ${\mathfrak Q}$.

\end{example}

}

Throughout the paper, we use the following notation. 
By convention, we use $\R^n$ to represent  $n$ dimensional Euclidean space equipped with the Euclidean norm $\|\cdot\|$,
$d(z_1,z_2):=\|z_1-z_2\|$ 
to denote 
the distance from a point $z_1$ to another point $z_2$
and $\R^n_+$ the first orthant of $\R^n$ with non-negative components.
{\color{black} For a given positive integer $N$}, we write $[N]$ and $[N_0]$ 
for $\{1,\cdots,N\}$ and
$\{0,1,\cdots,N\}$ respectively.
For a given real number $a$, 
we denote the indicator function of an interval 
$[a,1]$
by $\mathbbm{1}_{[a,1]}(t)$, 
which takes a value of $1$ for 
$t\in[a,1]$ and $=0$ otherwise. 
For $p\in [1,\infty)$, we let $\mathscr{L}^p[0,1]$ denote the 
$\mathscr{L}^p$-space of measurable functions $h:[0,1]\rightarrow\R$ with finite $p$th order moments, i.e.,
$\mathscr{L}^p[0,1]:=\{h:[0,1]\rightarrow\R\left|\int_{0}^1|h(t)|^p dt<\infty\right.\}$. Moreover, we let 
$\mathscr{P}(S)$ and $\mathscr{P}(S^N)$ denote the set of probability distributions with support set $S$ and  
$S^N:=\{s^1,\cdots,s^N\}$ respectively.


\section{Computation of the random spectral risk measures}


We begin with a basic assumption for the setup of the $\inmat{RSRM}$.

\begin{assumption}
The DM's risk preference
can be represented by an SRM but there does not exist
a deterministic risk spectrum $\sigma$ such that
the DM's risk preference can be described by $\rho_\sigma$.
\label{Assu:basic-rsmr}
\end{assumption}

\subsection{Random parametric SRM}



Under Assumption \ref{Assu:basic-rsmr},
we propose to use a $\inmat{RSRM}$ $\rho_{\sigma(\cdot,s)}$ to describe the DM's risk preferences. A key question is how to identify
an appropriate $\inmat{RSRM}$ which fits into the DM's risk preferences. We begin by considering some specific
parametric SRMs where the DM's risk preferences may be described by any one of them with the underlying  parameters being  randomized properly.


\begin{example}\label{exm:rsrm}
Let $X$ be a random variable representing financial losses.

\begin{itemize}
\item [\inmat{($i$)}] $\inmat{(Randomized CVaR)}$
Let $\alpha$ be a random variable taking values over
$(0,1)$ and
\begin{equation}
\label{eq:step-RS-CVaR}
  \sigma(\tau,\alpha) := \frac{1}{1-\alpha}\mathbbm{1}_{[\alpha,1]}(\tau),  
\end{equation}
where $\mathbbm{1}_{[a,b]}(\cdot)$ denotes the indicator function over interval $[a,b]$. Then
$$
\rho_{\sigma(\cdot,\alpha)}(X) = \inmat{CVaR}_\alpha(X).
$$
In this case, the $\inmat{DR-ARSRM}$ of $X$
can be written as
\bgeqn
   \rho_{{\mathfrak{Q}}}(X):= \sup_{Q\in {\mathfrak{Q}}}
  \mathbb{E}_Q\left[\rho_{\sigma(\cdot,\alpha)}(X)\right]
  = \sup_{Q\in {\mathfrak{Q}}}
  \mathbb{E}_Q\left[
 \inmat{CVaR}_\alpha(X)
 \right],
   \label{eq:Exp-RSRM-R-CVaR}
\edeqn
where $\mathfrak{Q}$ is an ambiguity set of 
probability
distributions of $\alpha$. The rhs of 
(\ref{eq:Exp-RSRM-R-CVaR}) coincides with Kusuoka's representation (\ref{eq:kusuoka's-representation})
when $\mathfrak{Q}=\mathcal{M}$.
We call $\mathcal{M}$ {\em Kusuoka's ambiguity set}.
In Section \ref{construction-Kasuoka's-set}, we will come back to discuss how the ambiguity set may be constructed
in a preference elicitation process. 

\item [\inmat{($ii$)}]
$\inmat{(Randomized Wang's proportional hazards transform or power distortion)}$
Let $r$ be a random variable taking values over
$(0,1]$
and
\begin{equation}
    \label{eq:wang's}
\sigma(\tau,r) = r(1-\tau)^{r-1}.
\end{equation}
Then
$
 \rho_{\sigma(\cdot,r)}(X) = \int_{0}^{\infty}(1-F_{X}(t))^rdt.
$
In this case, 
the $\inmat{DR-ARSRM}$ of $X$
can be written as
\begin{equation*}
    \rho_{{\mathfrak{Q}}}(X):= \sup_{Q\in {\mathfrak{Q}}}
  \mathbb{E}_Q\left[\rho_{\sigma(\cdot,r)}(X)\right]
  = \sup_{Q\in {\mathfrak{Q}}}
  \mathbb{E}_Q\left[
 \int_{0}^{\infty}(1-F_{X}(t))^rdt
 \right],
   \label{eq:Exp-RSRM-R-Wang}
\end{equation*}
where $\mathfrak{Q}$ is an ambiguity set of the probability
distributions of $r$.

\item [\inmat{($iii$)}] $\inmat{(Gini's measure)}$
Let $s$ be a random variable taking values over $(0,1)$ 
and
\begin{equation}
\label{eq:Gini}
    \sigma(\tau,s) = (1-s) +2s\tau.
\end{equation}
Then
$$
  \rho_{\sigma(\cdot,s)}(X)  =\mathbb{E}[X]+s\mathbb{E}(|X-X'|)=: \inmat{Gini}_s(X).
$$
In this case, the $\inmat{DR-ARSRM}$ of $X$
can be written as
$$
   \rho_{{\mathfrak{Q}}}(X):= \sup_{Q\in {\mathfrak{Q}}}
  \mathbb{E}_Q\left[\rho_{\sigma(\cdot,s)}(X)\right]
  = \sup_{Q\in {\mathfrak{Q}}}
  \mathbb{E}_Q\left[
 \inmat{Gini}_s(X)
 \right],
   \label{eq:Exp-RSRM-R-Gini}
$$
where $\mathfrak{Q}$ is an ambiguity set of the probability
distributions of $s$.

\item [\inmat{($iv$)}]
$\inmat{(Convex combination of expected value and CVaR risk measure)}$
Let $\alpha$ and $\lambda$ be two independent
random variables taking values over $(0,1)$ and
\begin{equation}
    \label{eq:SRM-convex-combination}
\sigma(\tau,\lambda,\alpha)=\lambda \mathbbm{1}_{[0,1]}(\tau)+(1-\lambda)\frac{1}{1-\alpha}\mathbbm{1}_{[\alpha,1]}(\tau),
\end{equation}
where $\mathbbm{1}_{[a,b]}$ denote the indicator function over interval $[a,b]$.
Then
$$
  \rho_{\sigma(\cdot,\lambda,\alpha)}(X) =\lambda \mathbb{E}[X]+(1-\lambda)\inmat{CVaR}_{\alpha}(X)
$$
defines a randomized convex combination of
the expected value and randomized CVaR risk measure.
In this case, the $\inmat{DR-ARSRM}$ of $X$
can be written as
\begin{equation*}
    \rho_{{\mathfrak{Q}}}(X):= \sup_{Q\in {\mathfrak{Q}}}
  \mathbb{E}_Q\left[\rho_{\sigma(\cdot,s)}(X)\right]
  = \sup_{Q\in {\mathfrak{Q}}}
  \mathbb{E}_Q\left[
\lambda \mathbb{E}[X]+(1-\lambda)\inmat{CVaR}_{\alpha}(X)
 \right],
   \label{eq:Exp-RSRM-R-CECVaR}
\end{equation*}
where $\mathfrak{Q}$ is an ambiguity set of the joint probability
distributions of $\lambda$ and $\alpha$.
\end{itemize}
\end{example}

\subsection{Computation of ARSRM}
\label{sec:calmeanofsrm}


Let X be 
finitely distributed with $\mathbb{P}(X=x_k)=p_k$ 
where $-\infty=x_0<x_1<x_2<\cdots<x_K<x_{K+1}=+\infty$. Let  $\pi_k:=\sum_{i=1}^{k-1}p_i$ for $k\in[K+1]$, where $\pi_1 =0$ and $\pi_{K+1}=1$. The 
cumulative distribution function (cdf)
of $X$ is a non-decreasing and right-hand side continuous step-like function with $K$ breakpoints, that is
\begin{equation*}
    F_X(x) = \pi_k, \ \text{for} \ x\in[x_{k-1},x_{k}), k\in[K+1].
\end{equation*}
The 
quantile function of $X$ is a non-decreasing and left-hand side continuous step-like function with $K$ breakpionts, i.e.,
\begin{equation*}
    F^\leftarrow_X(t) = x_k, \ \text{for} \ t\in (\pi_k,\pi_{k+1}], k \in[K].
\end{equation*}
Let $\mathfrak{S}(S) :=\{\sigma(\cdot,s)\in\mathfrak{S}:  s\in S\}$ 
be a set of 
random risk spectra.
For fixed $s\in S$, 
the discrete distribution of $X$
allows us to 
simplify the representation of the 
RSRM of $X$:
\begin{equation*}
    \rho_{\sigma(\cdot,s)}(X) = \int_{0}^1F^\leftarrow_X(t)\sigma(t,s)dt
     = \sum_{k=1}^K x_k\int_{\pi_k}^{\pi_{k+1}}\sigma(t,s)dt
     = \sum_{k=1}^K x_k\psi_k(s),
      \label{eq:rnd-srm-dis}
\end{equation*}
where $\psi_{k}(s)= \int_{\pi_k}^{\pi_{k+1}}\sigma(t,s)dt$ and $\sum_{k=1}^K\psi_{k}(s)=1$. 
Since $\pi_k$ is the cdf of $X$,
this formulation 
relies heavily on the order of
the outcomes of $X$.
Let $Q$ denote the
distribution of $s$. Then 
the ARSRM of $X$ (see (\ref{eq:Exp-RSRM})) can be formulated as
\begin{equation}
    \rho_Q(X)=\int_{S}\sum_{k=1}^Kx_k\psi_{k}(s)Q(ds)
    = \sum_{k=1}^K
    x_k\int_S\psi_{k}(s)Q(ds) =
    \sum_{k=1}^K x_k
    \int_{\pi_k}^{\pi_{k+1}}
    \int_S\sigma(t,s)Q(ds)dt.
\end{equation}
Moreover, by setting
$a(t):=\int_S\sigma(t,s)Q(ds)$, we can write $\rho_Q(X)$
succinctly as
\begin{equation*}
    \rho_Q(X)=
        \sum_{k=1}^K x_k
    \int_{\pi_k}^{\pi_{k+1}}a(t)dt=  \sum_{k=1}^K x_k a_k,
    \label{eq:avg-srm-dis}
\end{equation*}
where $a_k:=\int_{\pi_k}^{\pi_{k+1}}a(t)dt$. 
Since
$a_1\leq a_2\leq \cdots\leq a_{K}$ and $\sum_{k=1}^{K} a_k=1$, this manifests that
$\rho_Q(\cdot)$ is a coherent risk measure (see \cite[Theorem 4.1]{Ace02}). It also shows that $\rho_Q(\cdot)$  
is a specific spectral risk measure. The average effect of $\inmat{RSRM}$ is equivalent to some deterministic SRM although we do not know the latter.



\subsection{Random step-like risk spectrum}
\label{sec:randomSRS}


In practice, it is often desirable
to consider step-like risk spectrum.
Indeed, from (\ref{eq:rnd-srm-dis}), we can see that $ \rho_{\sigma(\cdot,s)}(X)$ is uniquely determined by $\psi_k(s)$, $k\in[K]$ and this is equivalent to setting
$$
\sigma(t,s) = \psi_k(s)/(\pi_{k+1}-\pi_k), \; \inmat{for} \; t\in [\pi_k,\pi_{k+1}).
$$
Several papers \cite{PiS15,wx20,gx21a} have shown that when the number of breakpoints  is sufficiently large, the step-like function can efficiently approximate the real risk spectrum function. In addition, when the risk spectrum is a step-like function, RSRM and {\color{black}the corresponding optimization problem} can be easily solved. 
Let $\mathfrak{S}_M$ denote the set of all nonnegative, non-decreasing and normalized step-like functions over $[0,1]$ with breakpoints
$$0=t_0<t_1<\cdots,t_M<t_{M+1}=1.$$
{\color{black}Let $\sigma_{M}(t,s):=\sum_{i=0}^M\sigma_{i}(s)\mathbbm{1}_{[t_i,t_{i+1})}(t)$ be a random step-like
risk spectrum.}
It obvious that $\sigma_{M}(t,s)\in\mathfrak{S}_M(S)$, where $\mathfrak{S}_M(S):=\{\sigma_{M}(t,s)\in\mathfrak{S}_M:s\in S\}$.
Moreover, let 
$$\Phi=\left[\left(\int_{\pi_k}^{\pi_{k+1}}\mathbbm{1}_{[t_i,t_{i+1})}(t)dt\right)_{k,i}\right]_{[K]\times[M_0]}\in\R^{K\times(M+1)}, \ X_K=(x_1,\cdots,x_K)^\top\in\R^{K}$$
and 
$\boldsymbol{\sigma(s)}=(\sigma_0(s),\sigma_1(s),\cdots,\sigma_M(s))^\top\in\R^{M+1}.$
Then the random step-like spectral risk measure can be written as
\begin{equation*}\label{def:sSRM}
   \rho_{\sigma_{M}(\cdot,s)}(X) = \sum_{k=1}^K x_k\sum_{i=0}^M\left(\int_{\pi_k}^{\pi_{k+1}}\sigma_{i}(s)\mathbbm{1}_{[t_i,t_{i+1})}(t)dt\right)=X_K^\top\Phi\boldsymbol{\sigma(s)}.
\end{equation*} 
Note that the {\color{black}random} step-like risk spectrum is non-negative, non-decreasing function with the normalized property $\int_0^1\sigma_M(t,s)dt=1$ for all $s\in S$. 
Let
\begin{equation*}
    A :=\left(\begin{matrix}
    -1 & 1 &  0 & \cdots & 0 & 0\\
    0 &  -1 & 1 & \cdots & 0 & 0\\
\vdots&\vdots&\vdots&\cdots&\vdots&\vdots\\
    0 &  0 &  0 & \cdots & -1 & 1
    \end{matrix}\right)\in\R^{M\times (M+1)}
\quad \inmat{and} \quad    
 \tilde{t} = \left(\begin{matrix}
t_1-t_0\\
t_2-t_1\\
\vdots\\
t_{M+1}-t_M
\end{matrix}\right)\in\R^{M+1}.
\end{equation*}
Then the vector of parameters $\boldsymbol{\sigma(s)}$, which effectively randomizes the step-like risk spectrum, is supported by
\begin{equation*}
\label{def:step-likeRS-set}
    \mathfrak{S}_M(S):=\{\boldsymbol{\sigma(s)}\in\R^{M+1}|A\boldsymbol{\sigma(s)}\geq0,\tilde{t}^\top \boldsymbol{\sigma(s)}=1,\boldsymbol{\sigma(s)}\geq0,s\in S\}.
\end{equation*}
Let $Q$ denote the probability of $s$. Then the 
ARSRM with 
random step-like risk spectrum 
can be expressed by
 \bgeq 
 \rho_{Q}(X) &=& \int_S\rho_{\sigma_M(\cdot,s)}(X)Q(ds) =\int_S \left[\sum_{k=1}^K x_k\left(\sum_{i=0}^M\int_{\pi_k}^{\pi_{k+1}}
 \sigma_{i}(s)\mathbbm{1}_{[t_i,t_{i+1})}(t)dt \right)\right]Q(ds) \\
 &=& \sum_{k=1}^K x_k\sum_{i=0}^M\left(\int_S\sigma_{i}(s)Q(ds)\right)\int_{\pi_k}^{\pi_{k+1}}\mathbbm{1}_{[t_i,t_{i+1})}(t)dt\\
 &=&X_K^\top \Phi\mathbb{E}_Q[\boldsymbol{\sigma(s)}].
\edeq

\subsection{Practical construction
of
step-like random risk spectrum
}

In practice, we may only know the values of 
RSRM but not the random risk spectra.
For instance, we know different premiums that an insurance company charges
on different groups of insureds, but don't know the true random risk spectra which represents the insurance company's risk preferences. 
In this subsection, we consider the case that 
we are able to observe  $L$ values of the RSRM w.r.t.~random loss $X$, i.e., $\rho^1_{\sigma(\cdot,s)}(X),\cdots,
\rho^L_{\sigma(\cdot,s)}(X)$, for each $s\in S$
which are calculated with random risk spectrum $\sigma(\cdot,s)$ (for example, the $L$ values of the RSRM correspond to premium prices of $L$ contracts). 
Assume further that $S=\{s^1,\cdots,s^N\}$.
For each $\rho^l_{\sigma(\cdot,s)}(X)$, $\forall l\in[L]$, $s\in S$, we have an ordered sample $\{x_{1l},\cdots,x_{Kl}\}$ of size $K$ with $\mathbb{P}(X=x_{kl})=\frac{1}{K}$, $\forall k\in[K]$, that is,
$$
\rho^l_{\sigma(\cdot,s)}(X) = \sum_{k=1}^K x_{kl}\int_{\frac{k-1}{K}}^{\frac{k}{K}}\sigma(t,s)dt.
$$
Since $\sigma(t,s)$ is unknown, we propose to construct a step-like function to approximate it. Let $\sigma_K(t,s)$ be a step-like function with breakpoints $0,\frac{1}{K},\cdots,\frac{K-1}{K},1$, i.e.,
\begin{equation*}
    \sigma_K(t,s^i) = \varpi_{k,i}, \ \ \text{for} \ t\in\left[\frac{k-1}{K},\frac{k}{K}\right), k\in[K], i\in[N],
\end{equation*}
where $\{\varpi_{k,i}\}_{k=1}^K$ is a non-negative and non-decreasing sequence with the normalized property $\frac{1}{K}\sum_{k=1}^K\varpi_{k,i} = 1$ for all $i\in[N]$. The approximate RSRM with $\sigma_K(t,s^i)$ and $ \mathbb{P}$ can be calculated by
\begin{equation*}
    \rho^l_{\sigma_K(\cdot,s^i)}(X)
    =\frac{1}{K}\sum_{k=1}^Kx_{kl}\varpi_{k,i}.
\end{equation*}
Note that $\rho^l_{\sigma(\cdot,s^i)}(X)$ is the observed value of the RSRM based on the $l$th ordered sample $\{x_{1l},\cdots,x_{Kl}\}$ and $\rho^l_{\sigma_K(\cdot,s^i)}(X)$ is the approximate value of the RSRM based on the step-like approximation of $\sigma(t,s^i)$. We want to identify the values of parameters $\varpi_{k,i}$, $k\in[K],i\in[N]$ such that the two values are as close as possible and propose to do so by solving a regression problem
\begin{equation}\label{md:qre}
    \begin{split}
        \min_{\varpi_{k,i}}\quad&
        \sum_{i=1}^N\sum_{l=1}^L\left(\rho^l_{\sigma(\cdot,s^i)}(X)-\rho^l_{\sigma_K(\cdot,s^i)}(X)\right)^2\\
        \inmat{s.t.} \quad&\frac{1}{K}\sum_{k=1}^K
        \varpi_{k,i}=1, \ \forall i\in [N], \\
        & \varpi_{k,i}\leq\varpi_{k+1,i}, \ \forall \ k\in[K-1], i\in [N],\\
        &\varpi_{k,i} \geq 0, \ \forall k\in[K], i\in [N],
    \end{split} 
\end{equation}
    see a similar approach by Escobar and Pflug \cite{ep20}.

To visualize the results of the above discussions, we 
give 
graphical interpretations 
about this approach. Given 
three true risk spectra $\sigma(t,s_1)=\frac{1}{3}(1-t)^{-\frac{2}{3}}$, $\sigma(t,s_2)=t+\frac{1}{2}$ and $\sigma(t,s_3)=\frac{1}{2}(1-t)^{-\frac{1}{2}}$. We take the sample of $X$ randomly from uniform distribution on $(0,1)$. Next, we  compute $\varpi_{k,i}, \forall k\in[K], i\in[N]$ by solving problem (\ref{md:qre}), and subsequently  obtain 
step-like approximation of
the true risk spectrum. The detailed 
computational 
results are shown in Figure \ref{fig:breakponits}-\ref{fig:sample}.
\begin{figure}[H]
    \centering
    \includegraphics[scale=0.60]{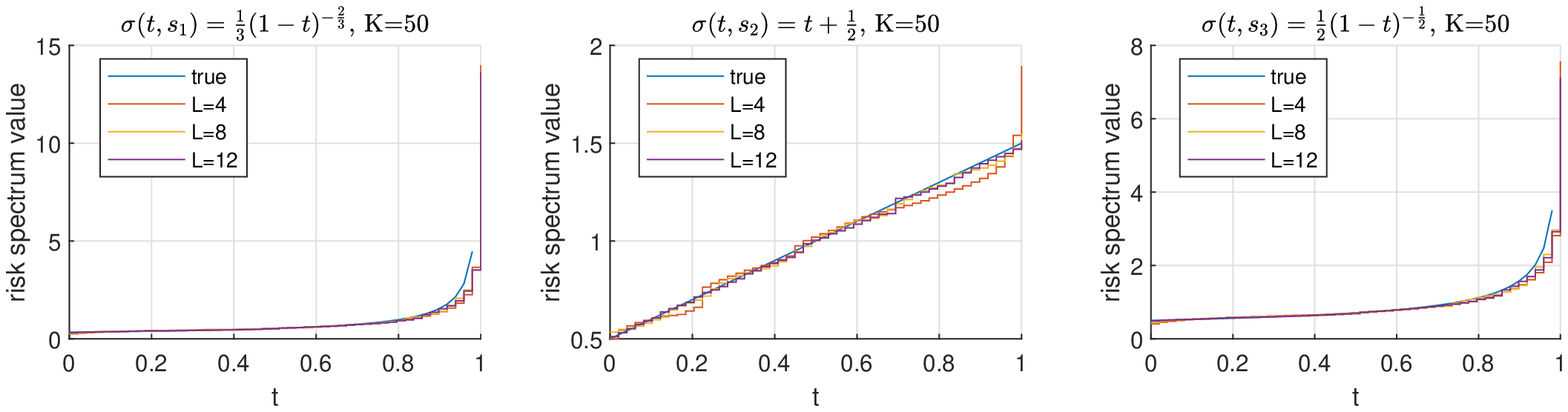}
    \caption{Approximate random risk spectrum of different number of observed values of the SRM}
    \label{fig:breakponits}
\end{figure}
\begin{figure}[H]
    \centering
    \includegraphics[scale=0.60]{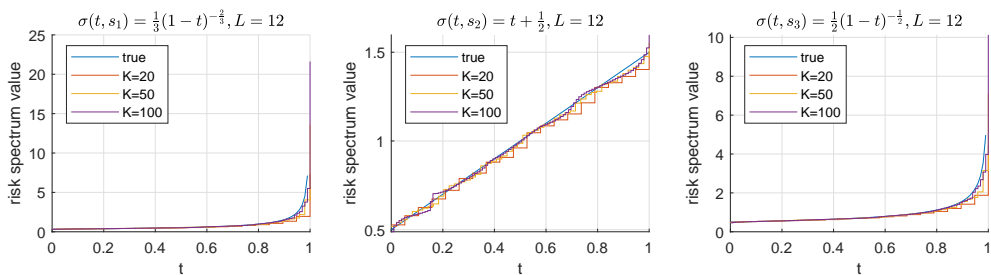}
    \caption{Approximate random risk spectrum of different number of breakpoints }
    \label{fig:sample}
\end{figure}

\section{Construction of the ambiguity set}
The ambiguity set $\mathfrak{Q}$ in (\ref{eq:Exp-RSRM-R-opti}) is a key element.
In this section, we give a sketch 
as to how it may be constructed using 
the well-known methods in the literature of distributional robust optimization (DRO).
For the Kusuoka' ambiguity set, we will propose a different approach.
\vspace{-0.2cm}
\subsection{Kantorovich ball approach}
 In the 
main stream research
of DRO models, the ambiguity is concerned with the probability distribution of exogenous/endogenous
uncertainty, see \cite{dy10,gs10,wks14,ek18,
ys20,Lam19} and the
references therein. Here, the ambiguity stems from incomplete information of probability distribution of endogenous uncertainty
(DM's risk preference). We consider a situation where
it is possible to use empirical data or subjective judgement to identify a nominal distribution $\hat{Q}$ but there is inadequate information to confirm whether $\hat{Q}$ is the true. Consequently, we construct 
an ambiguity set $\mathfrak{Q}$ by considering all probability distributions of $Q$ near $\hat{Q}$ 
using the well-known   Kantorovich/Wasserstein ball approach (\cite{ek18,PflugPichler2014,Gao22}). 
Let  $\{s^1,\cdots,s^N\}$ be 
independent
and identically distributed (i.i.d) copies of $s$, denoted by $S^N$, and $Q_N$ is a discrete empirical distribution.
Define the Kantorovich ball centered at $Q_N$:   
\begin{equation}
\setlength{\abovedisplayskip}{2pt}
\setlength{\belowdisplayskip}{2pt}
\label{def:Kantorovicball}
    {\mathfrak{Q}}_K(Q_N,r)=\left\{Q\in\mr{P}(S)|\mathsf{dl}_K(Q,Q_N)\leq r\right\},
\end{equation}
where $r$ is a positive number,
$\mathsf{dl}_K(\cdot,\cdot)$ is the Kantorovich metric defined by:
\begin{equation*}
\setlength{\abovedisplayskip}{2pt}
\setlength{\belowdisplayskip}{2pt}   \mathsf{dl}_K(Q_1,Q_2):=\inf_{\Pi\in\mathscr{P}(S\times S)}\int_{S\times S}\|s_1-s_2\|\Pi(ds_1\times ds_2),
\end{equation*}
where $\Pi$ is a joint distribution of $s_1$ and $s_2$ with marginals $Q_1$ and $Q_2$ respectively, $\|\cdot\|$ denotes an arbitrary norm defined on $S$. Note that Kantorovich metric can be reformulated as
\begin{equation*}
\setlength{\abovedisplayskip}{2pt}
\setlength{\belowdisplayskip}{2pt}    \mathsf{dl}_K(Q_1,Q_2)=\sup_{g\in\mathscr{G}}\left\{\int_{S}g(s)Q_1(ds)-\int_{S}g(s)Q_2(ds)\right\},
\end{equation*}
where $\mathscr{G}$ is the set of all Lipschitz functions with $|g(s_1)-g(s_2)| \leq \|s_1 -s_2\|$ for all $s_1,s_2\in\R$.
If the true probability distribution $Q^*$ of $s$ is light-tailed, that is, there exists a positive number $q>0$ such that
    $\delta:=\int_{\mathcal{S}}\exp(\|s\|^q)Q^*(ds)<\infty$,
then it follows by \cite[Theorem 2]{fg15} that for any $r>0$ and $N>1$, there exist positive constants $C_1$ and $C_2$ dependent only on $q$ and $\delta$  such that
\begin{equation}\label{pro:wass}
\setlength{\abovedisplayskip}{2pt}
\setlength{\belowdisplayskip}{2pt}  \prob(\mathsf{dl}_K(Q^*,Q_N)\geq r)\leq\left\{\begin{matrix}
        C_1 \exp(-C_2Nr^2), & \text{for} \ r\leq1,\\
        \exp(-C_2Nr^q), & \text{for} \ r>1,
    \end{matrix}\right.
\end{equation}
where ``Prob'' is the probability measure over the space $\R_+^N$ with Borel-sigma algebra $\mathcal{B}\otimes\cdots\otimes\mathcal{B}$ (N times).
We can estimate the radius of the smallest Kantorovich ball that contain $Q^*$ with confidence $1-\theta$ for certain $\theta\in(0,1)$. Let the rhs of (\ref{pro:wass}) equal to $\theta$. We can then figure out the radius of the
ball $r$ by
\begin{equation}\label{eq:rN}
    r_N(\theta)=\left\{
    \begin{matrix}        \left(\frac{\log(C_1\theta^{-1})}{C_2N}\right)^{1/2}, & \text{for} \ N\geq\frac{\log(C_1\theta^{-1})}{C_2},\\
        \left(\frac{\log(C_1\theta^{-1})}{C_2N}\right)^q, & \text{for} \ N<\frac{\log(C_1\theta^{-1})}{C_2},
    \end{matrix}\right.
\end{equation}
which means that the Kantorovich ball (\ref{def:Kantorovicball}) contains the true probability distribution $Q^*$ with
probability at least $1-\theta$ when  $r=r_N(\theta)$. Note that as the sample size N goes to infinity, the radius $r_N(\theta)$ converges to zero.
This means the Kantorovich ball converges to a singleton $Q^*$ as $N$ goes to infinity.

The Kantorovich ball defined in
(\ref{def:Kantorovicball}) contains 
both discrete and continuous distributions.
Differing from \cite[Theorem 4.1]{ek18},
the objective function 
in problem 
(\ref{eq:Exp-RSRM-R-opti})
is not convex in $s$ and hence, it is difficult to 
construct a tractable reformulation for 
the problem.
This prompts us to 
derive a discrete 
approximation of the Kantorovich ball. 

Let $\mathcal{S}^V:=\{\hat{s}^1,\cdots,\hat{s}^V\}$ be 
a pre-selected set of points in
$S$ and $\{S_1,\cdots,S_V\}$ be a Voronoi tessellation of $S$ centered at $\mathcal{S}^V$ , i.e.,
\begin{equation*}
 S_j\subseteq\left\{s\in S:\|s-\hat{s}^j\|=\min_{k\in[V]}\|s-\hat{s}^k\|\right\}, \text{for}~ j\in[V]
\label{eq:Voronoi-tessellation}   
\end{equation*}
are pairwise disjoint subsets forming a partition of $S$. 
Let $s^1,\cdots,s^N$ be an i.i.d sample of $s$
where the sample size $N$ is usually larger than $V$ and $N_i$ denote the number of samples falling into area $S_i$.
Define
\begin{equation}
\setlength{\abovedisplayskip}{2pt}
\setlength{\belowdisplayskip}{2pt}
\label{def:vt-projection}
Q_V(\cdot):=\sum_{j=1}^V\frac{N_i}{N}\delta_{\hat{s}^j}(\cdot),
\end{equation}
 where $\delta_{\hat{s}^j}(\cdot)$ denotes the Dirac probability measure located at $\hat{s}^j$. 
 Let  $\mathscr{P}(\mathcal{S}^V)$ be the set of all 
 probability distributions supported by $\mathcal{S}^V$.
 Note that $Q_V$ is the Voronoi projection of $Q_N$ onto space $\mathscr{P}(\mathcal{S}^V)$. 
We define a Kantorovich ball
in the space of $\mathscr{P}(\mathcal{S}^V)$ centered at $Q_V$
\begin{equation}
\setlength{\abovedisplayskip}{2pt}
\setlength{\belowdisplayskip}{2pt}
\label{eq:Kantorovichball-finite}
\mathfrak{Q}_K^V(Q_V,r)=\left\{Q\in\mathscr{P}(\mathcal{S}^V)|\mathsf{dl}_K(Q,Q_V)\leq r\right\}.
\end{equation}
We use $\mathfrak{Q}_K^V(Q_V,r)$
to approximate $\mathfrak{Q}_K(Q_N,r)$.

\subsection{Construction of Kusuoka's ambiguity set}
\label{construction-Kasuoka's-set}
In the literature of behavioural economics and preference robust optimization, a widely used approach
for eliciting a DM’s preference is pairwise comparison, that is, the DM is given
a pair of prospects for choice and the outcome of the choice is used to infer the DM’s
utility/risk preference. Here we use the same approach for constructing  Kusuoka's ambiguity set. 
In order for this approach to work, we need the following assumption.

\begin{assumption}
\label{asm:Kasuoka-ambiguityset}
Let $\mathscr{M}$ denote the set
of probability measures over $[0,1]$.
There exists a probability measure $\mu \in \mathscr{M}$
such that the DM's risk preference can be characterized by
$\int_0^1 \inmat{CVaR}_\alpha(\cdot)d\mu (\alpha)$.
\end{assumption}

The assumption is justified
in the case where 
a DM's preference can be used by a randomized $\inmat{CVaR}_\alpha(\cdot)$ but
there is short of complete information
on the distribution of $\alpha$.
Since 
\bgeqn
\label{def:kusuoka}
\left\{
\rho_\mu =\int_0^1 \inmat{CVaR}_\alpha(\cdot)d\mu (\alpha):
\mu\in \mathscr{M}
\right\}
\edeqn
forms a subset of LICRMs, 
the results established in this subsection 
cannot be applied to the case that 
a DM's risk preference 
is described by a general LICRM. Note that $\rho_\mu$ defined as \eqref{def:kusuoka} is a law invariant and comonotone coherent risk measure with the Fatou
property \cite[Theorem 7]{Kus01}.

Let $\{A_j,B_j\}_{j=1}^J$ be lotteries
for pairwise comparisons, that is,
the DM is given a pair $\{A_j,B_j\}$ and asked to select a preferable one.
Suppose that the DM is found to prefer $A_j$ over $B_j$
for $j\in[J]$. Based on  Assumption \ref{asm:Kasuoka-ambiguityset}, we can use the preference information to 
construct 
an ambiguity set of probability measures over $[0,1]$ as follows:
\bgeqn
\label{Pair-Kasuoka-ambiguityset}
\mathcal{M}_{pair}:=\left\{\mu\in\mathscr{M}:\int_0^1 \inmat{CVaR}_\alpha(A_j)d\mu (\alpha)\leq\int_0^1 \inmat{CVaR}_\alpha(B_j)d\mu (\alpha), \ j\in[J]\right\}.
\edeqn

\section{Tractable reformulations}
In this section, we discuss the tractable reformulations of (\ref{eq:Exp-RSRM-R}) under 
the ambiguity set $\mathfrak{Q}_K^V(Q_V,r)$. 
Moreover,
we investigate the reformulation of the Kasuoka's representation.

 \subsection{Tractable reformulations of the DR-ARSRM }
 
 We begin with tractable formulation of the $\inmat{DR-ARSRM}$ of discrete random loss $X$, i.e., 
$\mathbb{P}(X=x_k)=p_k$ for $k\in[K]$, where $x_1<x_2<\cdots<x_K$.
 
\begin{proposition}\label{thm:refm-phidivergence}
 The $\inmat{DR-ARSRM}$ of $X$ 
 under the Kantorovich ambiguity set $\mathfrak{Q}_K^V(Q_V,r)$ equals to the optimal value of the linear programming problem:
 \begin{equation*}
 \begin{split}
     \sup_{\boldsymbol{q},\boldsymbol{w}} \quad&\sum_{i=1}^V q_i\sum_{k=1}^K x_k\psi_{i,k}\\
     \inmat{s.t.}\quad&\sum_{j=1}^Vw_{ij}=q_i,\sum_{i=1}^Vw_{ij}={\color{black}\frac{N_j}{N}}, ~ \forall ~  i,j\in[V],\\
     & \sum_{i=1}^V\sum_{j=1}^Vw_{ij}|\hat{s}^i-\hat{s}^j|\leq r, \sum_{i=1}^Vq_i = 1,\\
    &w_{ij}\geq 0,~q_i\geq 0, ~ \forall  ~  i,j\in[V],
     \end{split}
\label{eq:DRSRM-Kanto-reformulate}
 \end{equation*}
 where $\psi_{i,k}=\int_{\pi_k}^{\pi_{k+1}}\sigma(t,\hat{s}^i)dt$ and $\pi_k=\sum_{i=1}^{k-1}p_i$ for all $i\in[V],k\in[K]$.
\end{proposition}

The conclusion of Proposition \ref{thm:refm-phidivergence} can be easily obtained by substituting ambiguity set 
$\mathfrak{Q}_K^V(Q_V,r)$ 
into  the $\inmat{DR-ARSRM}$ problem (\ref{eq:Exp-RSRM-R}). 


\subsection{Tractable formulation of Kusuoka's representation}
\label{subsec:Kusuoka's-representation}
Let $\{A_j,B_j\}_{j=1}^J$ be
a set of lotteries for
pairwise comparisons with finite outcomes, i.e.,
$\mathbb{P}(A_j=a^j_{k})=p^{j}_k$ and $\mathbb{P}(B_j=b^j_{k})=q^{j}_k$ for $k\in[K]$ and $j\in[J]$, where $a^j_{1}<a^j_{2}<\cdots<a^j_{K}$ and $b^j_{1}<b^j_{2}<\cdots<b^j_{K}$.
Following our discussions in Section \ref{sec:calmeanofsrm},
\begin{equation*}
\begin{split}
\int_{0}^1\inmat{CVaR}_\alpha(X)\mu(d\alpha) 
&= \int_{0}^1\sum_{k=1}^K x_k\left(\frac{1}{1-\alpha}\int_{\pi_k}^{\pi_{k+1}}\mathbbm{1}_{[\alpha,1]}(t)dt\right)\mu(d\alpha)\\
&= \sum_{k=1}^K x_k \int_{\pi_k}^{\pi_{k+1}} \left(\int_{0}^1\frac{1}{1-\alpha}\mathbbm{1}_{[\alpha,1]}(t)\mu(d\alpha)\right)dt\\
&=\sum_{k=1}^K x_k \int_{\pi_k}^{\pi_{k+1}}g(t) dt,
\end{split}
\end{equation*}
where $g(t)=\int_{0}^1\frac{1}{1-\alpha}\mathbbm{1}_{[\alpha,1]}(t)\mu(d\alpha)$. 
We consider the case that 
$\alpha$ 
has a finite support, denoted by
$\mathcal{A}=\{\alpha_1,\cdots,\alpha_N\}$, where $\alpha_1<\cdots<\alpha_N$. Let $m_i=\mu(\alpha=\alpha_i)$ for $i\in[N]$. Under Assumption~\ref{asm:Kasuoka-ambiguityset}, the DM's risk preference can be represented by 
$\rho(\cdot) = \sum_{i=1}^N m_i\inmat{CVaR}_{\alpha_i}(\cdot)$
in this case.
Moreover, 
we have $g(t)=\sum_{i=1}^N \frac{m_i}{1-\alpha_i}\mathbbm{1}_{[\alpha_i,1]}(t)$,
\begin{equation*}
   \int_{0}^1\inmat{CVaR}_\alpha(X)\mu(d\alpha) = \sum_{k=1}^K x_k\sum_{i=1}^N\frac{m_i}{1-\alpha_i}\int_{\pi_k}^{\pi_{k+1}}\mathbbm{1}_{[\alpha_i,1]}(t)dt=X_K^\top\Pi\boldsymbol{m}, 
\end{equation*}
\begin{equation*}
    \int_{0}^1\inmat{CVaR}_\alpha(A_j)\mu(d\alpha) = \sum_{k=1}^K x_k\sum_{i=1}^N\frac{m_i}{1-\alpha_i}\int_{\gamma^j_k}^{\gamma^j_{k+1}}\mathbbm{1}_{[\alpha_i,1]}(t)dt=A^{j\top}_K C^j \boldsymbol{m}
\end{equation*}
and
\begin{equation*}
    \int_{0}^1\inmat{CVaR}_\alpha(B_j)\mu(d\alpha) = \sum_{k=1}^K x_k\sum_{i=1}^N\frac{m_i}{1-\alpha_i}\int_{\eta^j_k}^{\eta^j_{k+1}}\mathbbm{1}_{[\alpha_i,1]}(t)dt=B^{j\top}_K D^j \boldsymbol{m},
\end{equation*}
where $\boldsymbol{m}=(m_1,\cdots,m_N)^\top\in\R_+^N$, 
$\Pi, C^j, D^j\in\R^{K\times N}$, 
$\Pi_{ki} = \frac{1}{1-\alpha_i}\int_{\pi_k}^{\pi_{k+1}}\mathbbm{1}_{[\alpha_i,1]}(t)dt$, 
$C^j_{ki} = \frac{1}{1-\alpha_i}\int_{\gamma^j_k}^{\gamma^j_{k+1}}\mathbbm{1}_{[\alpha_i,1]}(t)dt$,
$D^j_{ki} = \frac{1}{1-\alpha_i}\int_{\eta^j_k}^{\eta^j_{k+1}}\mathbbm{1}_{[\alpha_i,1]}(t)dt$, $\gamma_{k}^j=\sum_{i=0}^{k-1}p_k^j$ and $\eta_{k}^j=\sum_{i=0}^{k-1}q_k^j$ 
for all $k\in[K]$ and $i\in[N]$. Subsequently
 $   \sup_{\mu\in\mathcal{M}_{pair}}\int_{0}^1\inmat{CVaR}_\alpha(X)\mu(d\alpha)
$
can be formulated as a linear program:
\begin{equation}
\label{Kusuoka-discrete}
    \begin{split}
        \sup_{\boldsymbol{m}\in\R^N_+} \quad&X_K^\top\Pi\boldsymbol{m}\\
        \st \quad& \sum_{i=1}^N m_i =1,\\
        & A^{j\top}_K C^j\boldsymbol{m} \leq B^{j\top}_K D^j \boldsymbol{m}, \  j\in[J].
    \end{split}
\end{equation}
In the next example, we explain how Kusuoka's ambiguity set may be reduced using the pairwise comparison preference elicitation approach.


\begin{example}
\label{exm:kuosuoka-ambiguityset}
Let $\mathcal{A}=\{0.33,0.66,0.99\}$ 
and assume that the true probability distribution of 
$\alpha$, denoted by $\mu^*$, satisfies 
$\mu^*(\alpha=0.33) =0.3$, $\mu^*(\alpha=0.66) =0.3$
and $\mu^*(\alpha=0.99) =0.4$.
Then the DM's true risk measure
can be written as
$$
\rho^*(\cdot)=0.3\inmat{CVaR}_{0.33}(\cdot)+0.3\inmat{CVaR}_{0.66}(\cdot)+0.4\inmat{CVaR}_{0.99}(\cdot).
$$
We will use $\rho^*(\cdot)$ to determine which lottery is preferred in pairwise comparison $($in other words, 
$\rho^*(\cdot)$ acts as the DM$)$.
To elicit the preference, we
ask the DM 
questions $\{A^j,B^j\}$, $j\in [J]$ for pairwise comparisons. 
Let 
$$\boldsymbol{u}^j = \left(\inmat{CVaR}_{0.33}(A^j),\inmat{CVaR}_{0.66}(A^j),\inmat{CVaR}_{0.99}(A^j) \right)$$ 
and 
$$\boldsymbol{v}^j = \left(\inmat{CVaR}_{0.33}(B^j),\inmat{CVaR}_{0.66}(B^j),\inmat{CVaR}_{0.99}(B^j) \right).$$ 
Then we can obtain the
feasible region of $\boldsymbol{m}$ after the observing
the answer to question $j$ 
$$
\mathcal{M}^{j} := 
\mathcal{M}^{j-1}\bigcap\left\{\begin{matrix}\{\boldsymbol{m}\in\R_+^3:\boldsymbol{u}^j\boldsymbol{m} \leq \boldsymbol{v}^j\boldsymbol{m}\} & \text{if} \ \rho^*(A^j)\leq\rho^*(B^j)\\
\{\boldsymbol{m}\in\R_+^3: \boldsymbol{u}^j\boldsymbol{m} \geq\boldsymbol{v}^j\boldsymbol{m}\} & \text{if} \ \rho^*(A^j)\geq\rho^*(B^j)
\end{matrix}\right\},
$$
where $\mathcal{M}^0 :=\{\boldsymbol{m}\in\R_+^3:\sum_{i=1}^3m_i=1\}$. 
In the $j$th question, we randomly sample $A^j$ following a uniform distribution over $(0,50)$ and $B^j$ following a normal distribution with mean $30$ and standard deviation $6$ of cardinality $K=100$. 
Specifically, 
we observe $\rho^*(A^1)\leq\rho^*(B^1)$ in the first question. 
Thus we obtain 
a cut 
$4.1718m_1-0.16m_2-2.3192m_3\geq0$ of $\mathcal{M}^0$ (shown as Figure \ref{fig:CA_Q1}). 
In the second question,  
we observe 
that $\rho^*(A^2)\geq\rho^*(B^2)$
and hence
construct a cut 
$-2.7486m_1+2.3195m_2+5.6129m_3\geq0$ of $\mathcal{M}^1$ (shown as Figure \ref{fig:CA_Q2}). 
\begin{figure}[H]
  \centering
  \subfigure[]{
    \label{fig:CA_Q1}
    \includegraphics[scale=0.45]{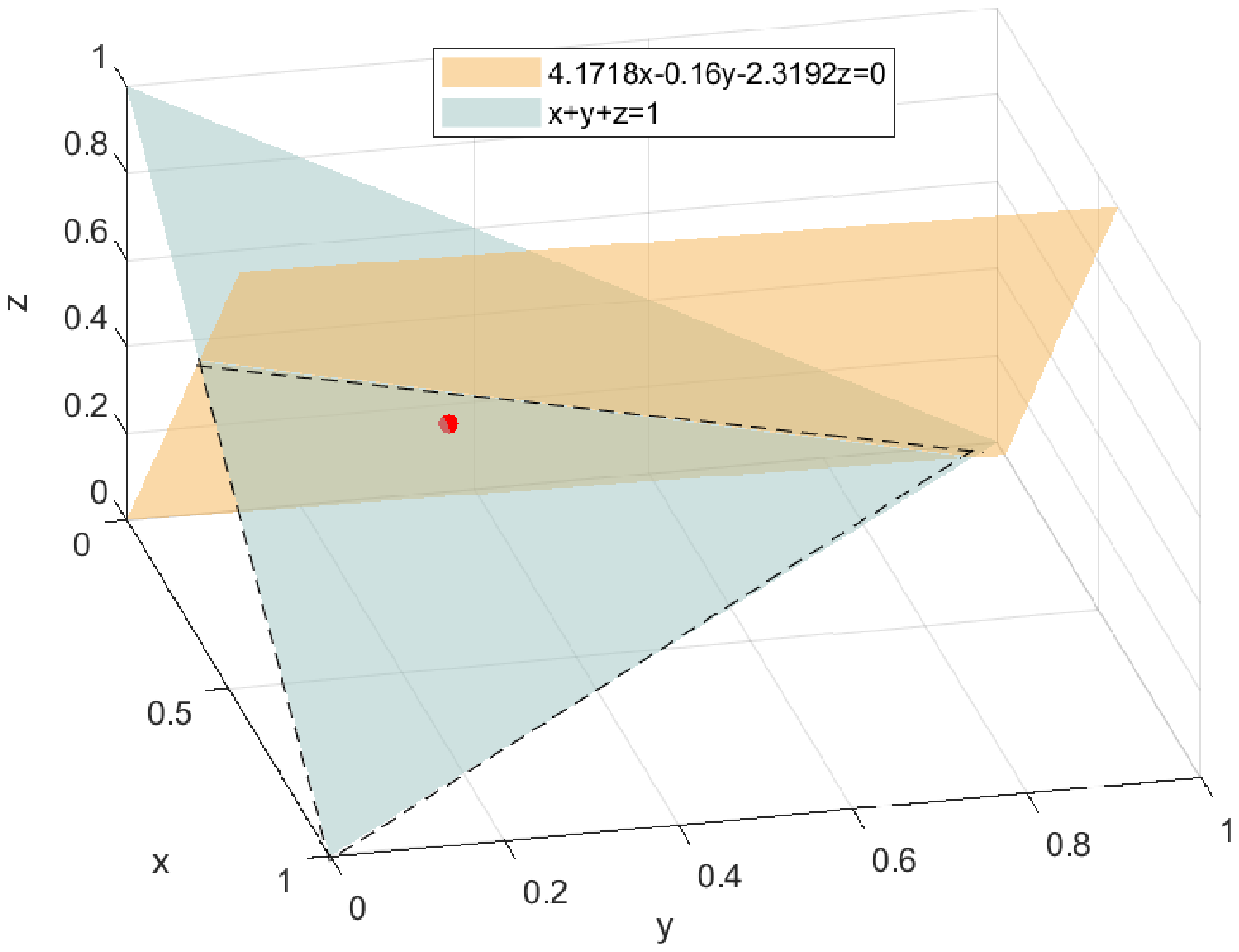}}
 \subfigure[]{
    \label{fig:CA_Q2}
    \includegraphics[scale=0.45]{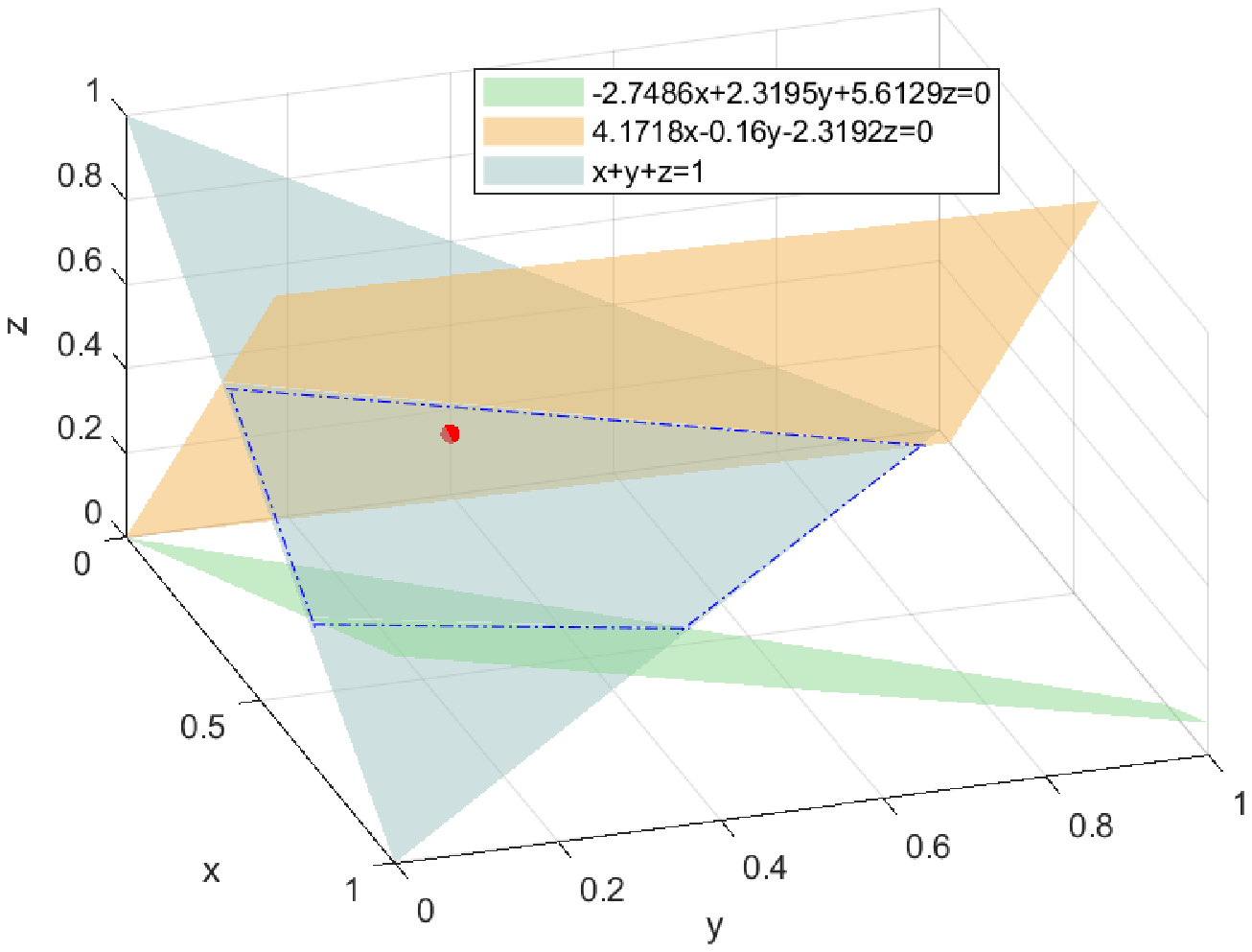}}
  \caption{(a) 
  The blue triangle 
 with  
black dashed boundaries represents $\mathcal{M}^1$. 
(b) The blue 
quadrangle with dark blue
dashed boundaries
represents $\mathcal{M}^2$.
The red dot 
  stands for $\boldsymbol{m}^*=(0.3,0.3,0.4)$.}
  \label{fig:section4-exam}
\end{figure}
\end{example}

$$
$$





\section{Optimal decision making based on DR-ARSRM}
In this section, we move to discuss
optimal decision making problems 
based on DR-ARSRM. Specifically, 
let $f(z,\xi)$ be a loss function where
$z$ is a decision vector and $\xi$ is a vector of random variables. 
We consider the following minimax optimization problem:
\begin{equation}\label{def:DR-RSRM-Opt}
\setlength{\abovedisplayskip}{2pt}
\setlength{\belowdisplayskip}{2pt}
\inmat{(DR-ARSRM-Opt)}\quad\min_{z\in Z} \ \max_{Q\in\mathfrak{Q}} \int_S \rho_{\sigma(\cdot,s)}(f(z,\xi))Q(ds), 
\end{equation}
where $\mathfrak{Q}\subset\mr{P}(S)$ denotes
the  ambiguity set of 
distributions of $s$,
$Z$ is a compact convex set 
and $f(z,\xi)$ is a convex function 
in $z$ for all $\xi$. 

To facilitate derivation of tractable formulations of the minimax problem, we restrict our discussions throughout this section to the case that $s$ is discretely distributed and $\sigma(\cdot,s)$ is a step-like function for each fixed $s$.
Specifically, we consider:
\begin{equation}
\setlength{\abovedisplayskip}{2pt}
\setlength{\belowdisplayskip}{2pt}
\inmat{(Appr-DR-ARSRM-Opt)}\quad\min_{z\in Z} \ \max_{\boldsymbol{q}\in\mathfrak{Q}^V} \sum_{i=1}^V
q_i\rho_{\sigma_M(\cdot,\hat{s}^i)}(f(z,\xi)),
\label{def:DR-RSRM-Opt-approx}
\end{equation}
where $\mathfrak{Q}^V\subset\mathscr{P}(\mathcal{S}^V)$ is the subset of 
distributions supported
by $\mathcal{S}^V=\{\hat{s}^1,\cdots,\hat{s}^V\}$
and $\sigma_M(\cdot,\hat{s}^i)$ is a step-like function over $[0,1]$ with breakpoints $ \{t_1,\cdots,t_M\}$, i.e.,  for fixed $\hat{s}^i\in \mathcal{S}^V$
\begin{equation*}
\setlength{\abovedisplayskip}{2pt}
\setlength{\belowdisplayskip}{2pt}
\sigma_{M}(t,\hat{s}^i)=\sum_{i=0}^M\sigma_i(\hat{s}^i)\mathbbm{1}_{[t_i,t_{i+1})}(t).
\end{equation*}
Since problem (\ref{def:DR-RSRM-Opt-approx}) may be regarded as an approximation of problem 
(\ref{def:DR-RSRM-Opt}), we denote it 
by Appr-DR-ARSRM-Opt.
In what follows, we study the computational schemes for solving (\ref{def:DR-RSRM-Opt-approx}) under the ambiguity set $\mathfrak{Q}_K^V(Q_V,r).$
  We propose
to solve the minimax optimization problem
by using an alternative interative 
scheme, i.e., for fixed $z$, solve the inner maximization problem and then for fixed $\boldsymbol{q}$, solve the outer minimization problem.

\noindent
\underline{\textbf{1. Solving the inner maximization problem}}.
Let  $z\in Z$ be fixed.
We consider a situation where $\xi$ is uniformly finitely distributed over $\{\xi^1,\cdots,\xi^K\}$ with $f(z,\xi^1)<f(z,\xi^2)<\cdots<f(z,\xi^K)$.  For a given set of step-like risk spectra $\mathfrak{S}_M(\mathcal{S}^V) := \{\sigma_M(\cdot,\hat{s}^1),\cdots,\sigma_M(\cdot,\hat{s}^V)\},$
 the inner maximization of  (\ref{def:DR-RSRM-Opt-approx}) under the ambiguity set $\mathfrak{Q}^V_K(Q_V,r)$
can be written 
as 
\begin{equation}\label{eq:DSRM-OPT-Kantorovich}
 \begin{split}
\sup_{\boldsymbol{q},\boldsymbol{w}} \quad&\sum_{i=1}^V q_i\sum_{k=1}^K f(z,\xi^k)\psi_{i,k}\\
\st\quad&\sum_{j=1}^Vw_{ij}=q_i, \sum_{i=1}^Vw_{ij}={\color{black}\frac{N_j}{N}}, ~ \forall ~  i,j\in[V],\\
&\sum_{i=1}^V\sum_{j=1}^Vw_{ij}|\hat{s}^i-\hat{s}^j|\leq r,\sum_{i=1}^Vq_i=1,\\
&w_{ij}\geq 0,~q_i\geq 0, ~ \forall  ~  i,j\in[V], 
\end{split}
\end{equation}
  via Proposition \ref{thm:refm-phidivergence}, where $\psi_{i,k}=\int_{\frac{k-1}{K}}^{\frac{k}{K}}\sigma_M(t,\hat{s}^i)dt=\sum_{j=0}^M\sigma_j(\hat{s}^i)\int_{\frac{k-1}{K}}^{\frac{k}{K}}\mathbbm{1}_{[t_i,t_{i+1})}(t)dt$ for all $i\in[V],k\in[K]$. 
Note that 
$\psi_{i,k}$ depends on not only $\hat{s}^i$ 
but also 
the cdf of $f(z,\xi)$. Thus the formulations require us to sort out the order of the outcomes of $f(z,\xi)$ in the first place for  fixed $z$.

\noindent
\underline{\textbf{2. Solving the outer minimization problem}}. 
Let $\boldsymbol{q}$ be fixed.
Following a similar strategy to that of 
\cite{wx21a}, we have
\begin{align}
    \rho_{\boldsymbol{q}}(f(z,\xi)) &= \sum_{i=1}^V q_i\sum_{k=1}^K f(z,\xi^k)\psi_{i,k}= \sum_{i=1}^V\sum_{k=1}^K q_i \left(\psi_{i,k}-\psi_{i,k-1}\right)\left(\sum_{j=k}^Kf(z,\xi^j)\right)\nonumber\\
    &=\sum_{i=1}^V\sum_{k=1}^K q_i (\psi_{k,i}-\psi_{k-1,i})(K-k+1)\left(\frac{1}{K-k+1}\sum_{j=k}^Kf(z,\xi^j)\right)\nonumber\\
    &= \sum_{i=1}^V\sum_{k=1}^K q_i\beta_{ik}\inmat{CVaR}_{\alpha_k}(f(z,\xi)),
\label{srm-re-cvar}
\end{align}
 where $\beta_{ik}=(\psi_{i,k}-\psi_{i,k-1})(K-k+1)$, $\alpha_k=\frac{k-1}{K}$ and $\psi_{i,0}=0$ for all $i\in[V],k\in[K]$.
we can recast it as:
\begin{equation}
\inmat{CVaR}_{\alpha}(X)=\min_{\eta\in\R} \ \eta+\frac{1}{1-\alpha}\mb{E}[(X-\eta)_+],
\label{def:CVaR-eq}
\end{equation}
where $(a)_+=\max\{a,0\}$. 
By substituting  (\ref{def:CVaR-eq}) into
(\ref{srm-re-cvar}),
we obtain
\begin{equation}
\label{eq:srm-Opt-CVaR}
    \min_{z\in Z}\rho_{\boldsymbol{q}}(f(z,\xi))=\min_{z\in Z,\boldsymbol{\eta}}\sum_{i=1}^V\sum_{k=1}^K q_i\beta_{ik}\left\{\eta_k+\frac{1}{1-\alpha_k}\mathbb{E}[(f(z,\xi)-\eta_k)_+]\right\}.
\end{equation}
A clear benefit of using the CVaR formulation is that we don't have to worry about 
the order of the outcomes 
of $f(z,\xi)$ in calculation of
ARSRM $\rho_{\boldsymbol{q}}(f(z,\xi))$ (see our comments in Section \ref{sec:calmeanofsrm}).
 Moreover, by introducing auxiliary variables,
 we can reformulate (\ref{eq:srm-Opt-CVaR}) 
 as
 \begin{equation}\label{eq:srm-Opt-CVaR-re}
 \begin{split}
\min_{z\in Z,\boldsymbol{\gamma},\boldsymbol{\eta},\boldsymbol{\zeta}}\quad&\sum_{i=1}^V\sum_{k=1}^K q_i\beta_{ik} \gamma_k\\
     \st\quad& \eta_k+\frac{1}{1-\alpha_k}\sum_{j=1}^K\frac{1}{K}\zeta_{jk}\leq\gamma_k, ~ 
     \forall~k\in[K],\\
     & f(z,\xi^j)-\eta_k\leq\zeta_{jk}, ~\forall ~j,k\in[K],\\
     &\zeta_{jk}\geq 0,~\forall~j,k\in[K].
 \end{split}
 \end{equation}
In the case that $f$ is linear in $z$, (\ref{eq:srm-Opt-CVaR-re}) is a linear programming problem. Note that this formulation is based on the assumption that the probability distribution of the state variable $s$,
$\boldsymbol{q}=(q_1,\cdots,q_V)^\top$,
is known. 

\noindent
\underline{\textbf{3. Solving problem (\ref{def:DR-RSRM-Opt-approx})}}. 
Based on the discussions above, we 
are ready to present an alternating iterative algorithm for solving 
(\ref{def:DR-RSRM-Opt-approx}).
\begin{algorithm}[H]
    \caption{An alternating iterative algorithm for solving (\ref{def:DR-RSRM-Opt-approx}) under ambiguity set (\ref{eq:Kantorovichball-finite})}
    \label{alm:DRSRM}
    \begin{algorithmic}
        \State {{\bf Initialization} Choose  initial values $z^1$ and $\boldsymbol{q}^1$. For $\ell=1,2,\cdots$}
        \Repeat
        \State {1. For fixed $z^\ell$, sort out function values $\{f(z^{\ell},\xi^k)\}_{k=1}^K$ in non-decreasing order and denote the sorted values by $\{f(z^{\ell},\xi_{(k)})\}_{k=1}^K$.} 
        Solve problem (\ref{eq:DSRM-OPT-Kantorovich}) with an optimal solution denoted  by $\boldsymbol{q}^{\ell+1}$.
        \State {2. For fixed $\boldsymbol{q}^{\ell+1}$, solve problem (\ref{eq:srm-Opt-CVaR-re}) with an optimal solution  $z^{\ell+1}$.}
        \Until{$z^{\ell+1}=z^{\ell}$ and $\boldsymbol{q}^{\ell+1}= \boldsymbol{q}^{\ell}$}.\\
        \Return{$\boldsymbol{q}^*:=\boldsymbol{q}^{\ell+1}, z^*:=z^{\ell+1}$}.
    \end{algorithmic}
\end{algorithm}
\begin{proposition}
\label{pro:algthm-convergence}
Algorithm \ref{alm:DRSRM} either terminates in a finite number of steps with a solution to the problem (\ref{def:DR-RSRM-Opt-approx}) or generates a sequence $\{(z^\ell,\boldsymbol{q}^\ell)\}$ whose cluster points are optimal solutions to the problem (\ref{def:DR-RSRM-Opt-approx}).
\end{proposition}
The proof is standard, we 
include a sketch of it 
in Appendix \ref{pro:algthm-convergence-proof}.

\section{Numerical tests}
We have carried out some numerical tests to validate (\ref{def:DR-RSRM-Opt-approx}) in the context of a stylized portfolio selection problem. All of the tests are carried out in Matlab 2021b installed in a Dell Intel Core i5 processor at 2.50 GHz and 8 GB of RAM and the optimization problems are solved 
by calling Gurobi 9.5.2 and fmincon solvers through the Yalmip interface. In this section, we report the details of the tests.

\subsection{Setup}

Consider a capital market consisting of $m$ assets whose yearly returns are captured
by the random vector $\xi  = (\xi_1,\cdots, \xi_m)^\top$. The total available funding is normalized to one and 
short-selling is not allowed.
Let  
$z = (z_1,\cdots,z_m)^\top$ 
denote the vector of allocations and 
$Z=\{z\in\R^m_+:\sum_{i=1}^m z_i=1\}$ the feasible set. 
For a given $z$, 
the total portfolio return is $z^\top\xi$ . 
We consider a case that the portfolio manager's  
risk attitude can be described by a random
spectral risk measure and base the optimal 
decision on the worst-case ARSRM:
\begin{equation}\label{eq:testproblem}
    \min_{z\in\Z}\max_{Q\in\mathfrak{Q}}\mathbb{E}_Q[\rho_{\sigma(\cdot,s)}(-z^\top\xi)].
\end{equation}
To solve the problem, we replace $\mathfrak{Q}$ with $\mathfrak{Q}^V$
and solve the approximate DR-ARSRM:
\begin{equation}\label{eq:testproblem-approx}
    \min_{z\in\Z}\max_{Q\in\mathfrak{Q}^V}\mathbb{E}_Q[\rho_{\sigma_M(\cdot,s)}(-z^\top\xi)],
\end{equation}
which is parallel to (\ref{def:DR-RSRM-Opt-approx}).
The tests are based on a market with $m=10$ assets as considered by Esfahani and Kuhn \cite{ek18}.
We follow them  to assume that the return $\xi_i$ is decomposed into a systematic risk factor $\varphi\sim\mathcal{N}(0,2\%)$ common to all assets and an unsystematic or idiosyncratic risk factor $\zeta_i\sim\mathcal{N}(i\times 3\%,i\times 2.5\%)$ specific to assets $i$ for $i\in\{1,\cdots,10\}$. Specifically we set $\xi_i=\varphi+\zeta_i$, where $\varphi $ and $\zeta_i$ constitute independent normal random variables. We generate $K$ i.i.d samples and select the breakpoints $t_1,\cdots,t_M$ from $\{\frac{1}{K},\frac{2}{K},\cdots,\frac{K-1}{K}\}$.
We consider two types of the
true random risk spectra: 
 {\bf(a)}
 $\sigma(t,s)=s(1-t)^{s-1}$ for $s\in (0,1)$ (randomized Wang’s proportional hazards transform), and 
 {\bf(b)} 
 $\sigma(t,\alpha)=\frac{1}{1-\alpha}\mathbbm{1}_{[\alpha,1]}(t)$ for $s\in (0,1)$. 
  For case (a), we 
consider step-like approximation:
    \begin{equation*}
         \sigma_M(t,s)=\left\{\begin{array}{lcl}
         \sigma(t_i,s), & & \inmat{for}\; t\in[t_i,t_{i+1}), i\in[M_0]\backslash\{M\},\\
          M+1-\sum_{i=0}^{M-1}\sigma(t_i,s),& & 
          \inmat{for}\;
          t\in[t_M,1).
         \end{array}\right.
     \end{equation*}

\subsection{
%
Randomized Wang’s proportional hazards transform}
{\color{black}

The first set of tests
concerns model (\ref{def:DR-RSRM-Opt}) with $\mathfrak{Q}_K^V(Q_V,r)$ being 
the discrete Kantorovich ambiguity set. In this case, the true 
risk spectrum takes a form of  
$\sigma(t,s)=s(1-t)^{s-1}$ for $s\in (0,1)$.

We begin by investigating the impact of 
the radius $r$ of the ambiguity set  
on the optimal value and optimal allocations/weights.
We solve problem (\ref{eq:testproblem-approx}) via Algorithm \ref{alm:DRSRM} using training datasets of cardinalities $N\in\{10,50,300\}$, $K=100$, the number of breakpoints $M=9$ with $t_i=\frac{i}{M+1}$ for $i\in[M_0]$ and the radius $r\in\{0,0.1,0.2,0.3,0.4,0.5,0.6,0.7,0.8,0.9,1\}$. Figures \ref{fig:Opt_portfolio_sample} and \ref{fig:Opt_value_sample}  depict the optimal portfolio wights $z(r)$ and the corresponding optimal value. The numerical results show that the optimal portfolio weights shift to the first six assets and they are becoming more and more evenly distributed.
This is perhaps because the DM becomes increasing conservative. 
The optimal value becomes larger as the radius $r$ increases, which confirms the fact that the range of the Kantorovich ball is lager when the radius $r$ increases, and thus the optimization problem (\ref{eq:testproblem-approx}) has a greater optimal value. 
}


\begin{figure}[H]
  \centering
  \subfigure[]{
    \label{fig:Opt_portfolio_sample10}
    \includegraphics[scale=0.4]{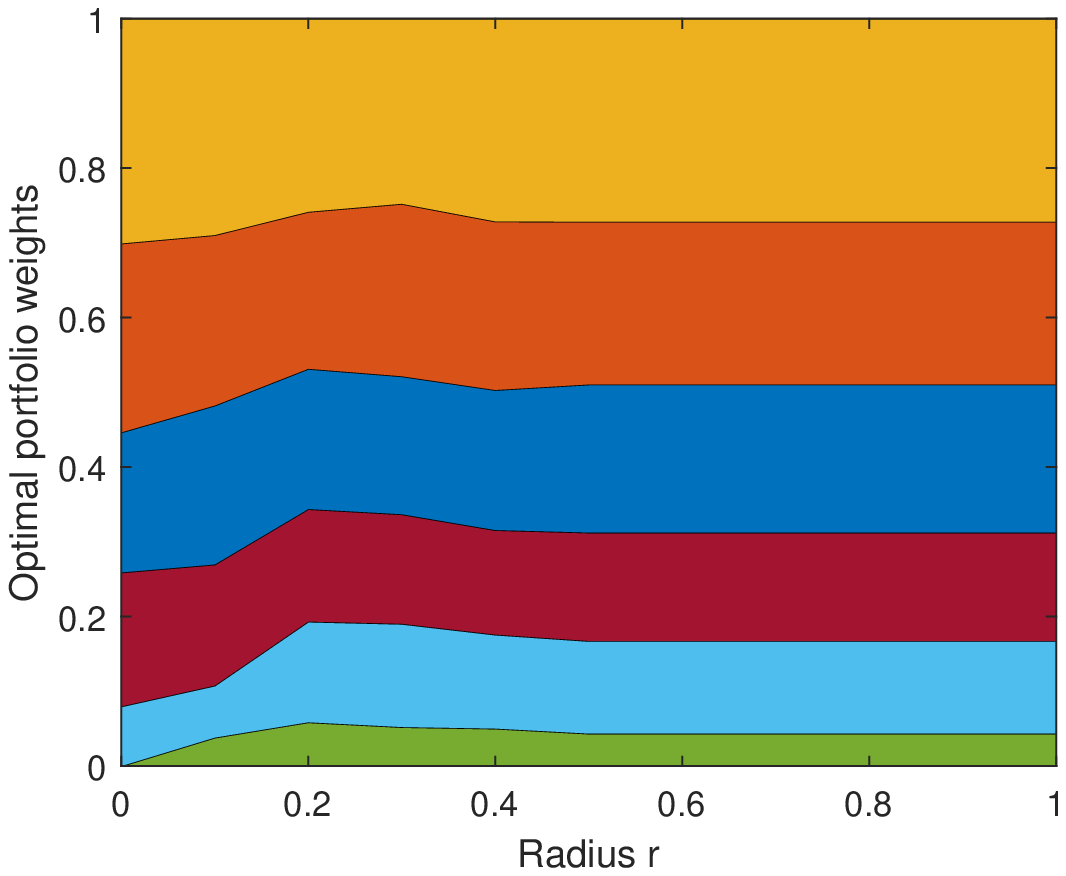}}
  \subfigure[]{
    \label{fig:Opt_portfolio_sample50}
    \includegraphics[scale=0.4]{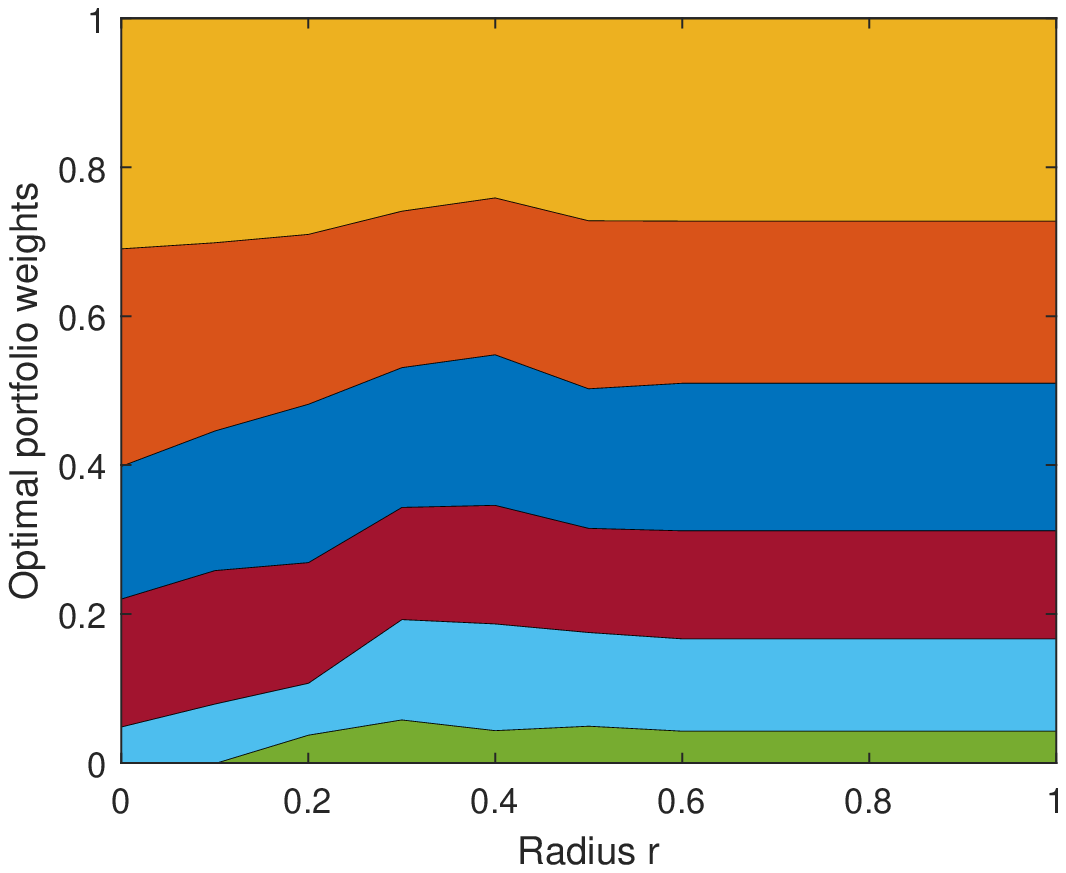}}
    \subfigure[]{
    \label{fig:Opt_portfolio_sample300}
    \includegraphics[scale=0.4]{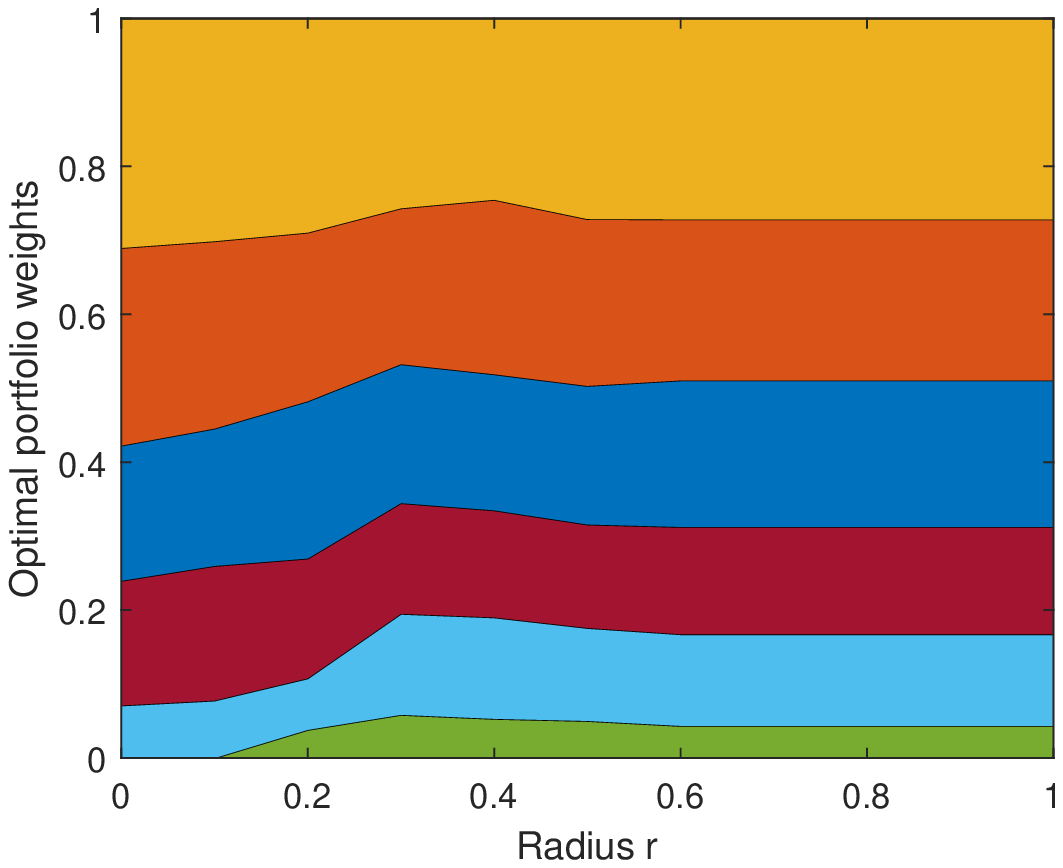}}
  \caption{{\color{black}Optimal portfolio composition as a function of the radius $r$ for $M=9$ breakpoints, $K=100$ sample sizes of $\xi$. (a) $N=10$ training samples. (b) $N=50 $ training samples. (c) $N=300$ training samples. The portfolio weights are visualized in ascending order, i.e., the top area corresponds to the most heavily weighted asset and the second from top corresponds to the second most heavily weighted asset, and so on.}}
  \label{fig:Opt_portfolio_sample}
\end{figure}
\begin{figure}[H]
  \centering
  \subfigure[]{
    \label{fig:Opt_value_sample10}
    \includegraphics[scale=0.4]{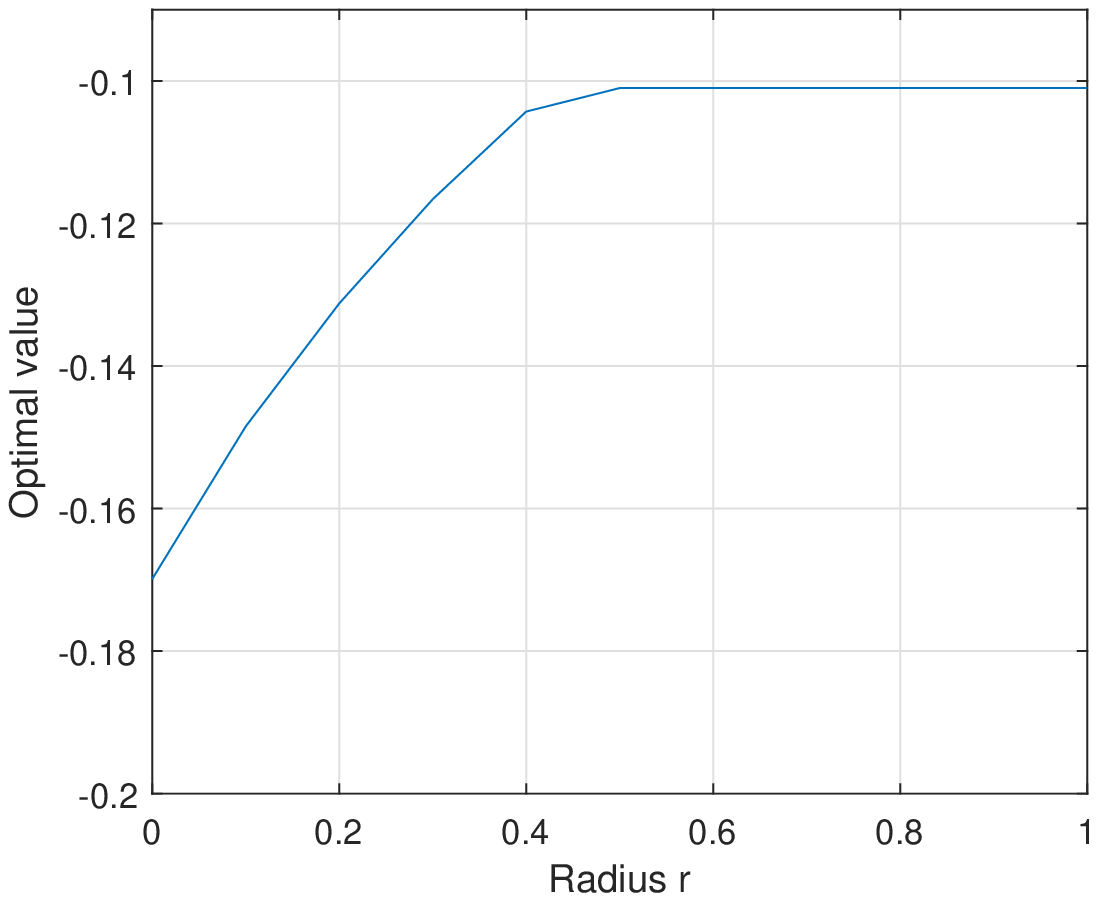}}
  \subfigure[]{
    \label{fig:Opt_value_sample50}
    \includegraphics[scale=0.4]{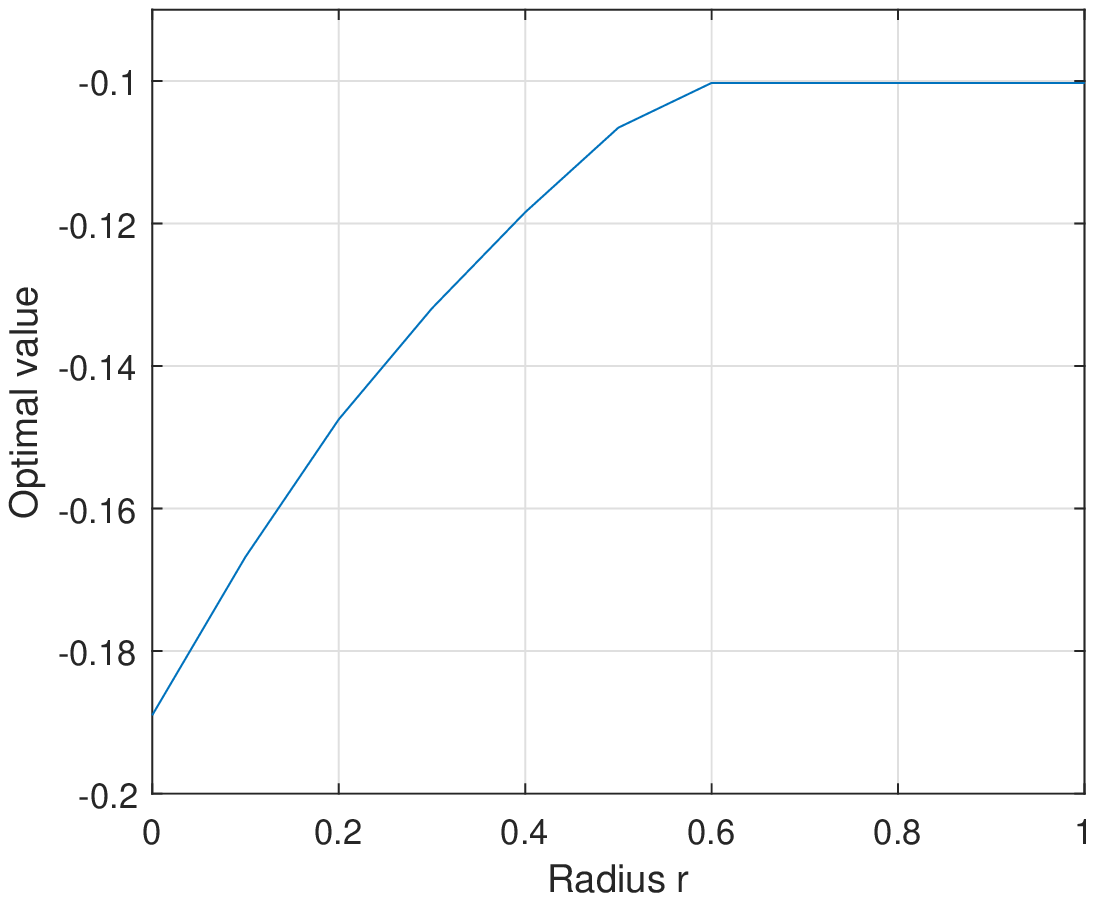}}
    \subfigure[]{
    \label{fig:Opt_value_sample300}
    \includegraphics[scale=0.4]{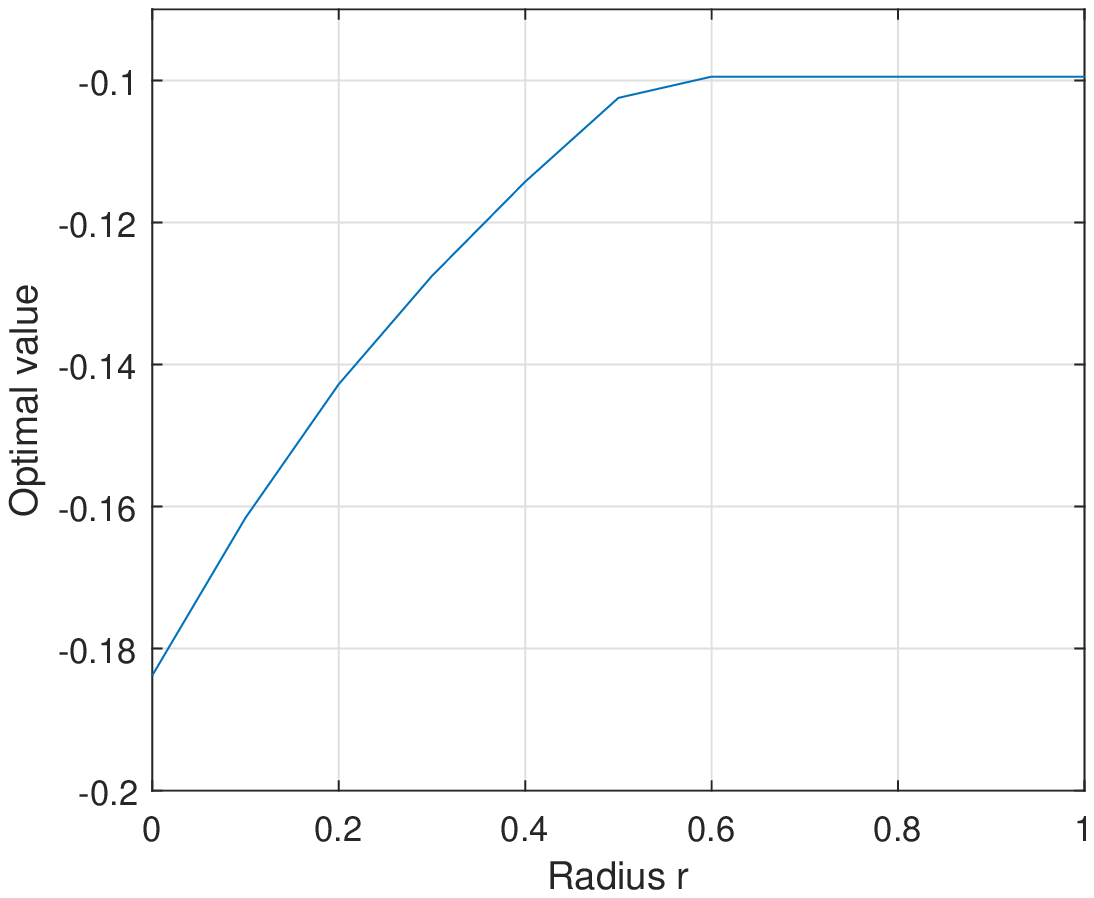}}
  \caption{{\color{black}Optimal value of (\ref{eq:testproblem-approx}) as a function of the radius $r$ for $M=9$ breakpoints, $K=100$ sample sizes of $\xi$. (a) $N=10$ training samples. (b) $N=50 $ training samples. (c) $N=300$ training samples.}}
  \label{fig:Opt_value_sample}
\end{figure}

Next, 
we investigate the impact on the optimal value 
and optimal portfolio weights 
when we use $\sigma_M(\cdot,s)$ to 
approximate $\sigma(\cdot,s)$ with increasing 
 number of breakpoints $M$.
We solve problem (\ref{eq:testproblem-approx}) via Algorithm \ref{alm:DRSRM} using training datasets of cardinalities $N=100$, $K=400$.
The radius of ambiguity set $r$ is set $0.2$, 
and number of breakpoints $M$ ranges over $\left\{1,3,7,9,15,19,24,39,49,\right.$
$\left.79,99,199,399\right\}$ 
such that $M+1$ is the divisor of $K$ and 
$t_i=\frac{i}{M+1}$ for $i\in[M_0]$. 
It is difficult to calculate the precise optimal value $\vartheta^*$  of (\ref{eq:testproblem})
since the inner maximization problem is infinite dimensional. In this experiment,
we use $N=300$ and $M=40000$ to solve the problem (\ref{eq:testproblem-approx}) and regard its optimal
value $\vartheta^*=-0.1015$ as the true optimal value of (\ref{eq:testproblem}).
Figures \ref{fig:Opt_portfolio_breakpoints} and \ref{fig:Opt_value_breakpoints} visualize the changes of the optimal portfolio weights $z$
and the optimal values as the number of breakpoints $M$ increases.
From Figure \ref{fig:Opt_value_breakpoints}, we can see that
the optimal value $\vartheta_M$ of (\ref{eq:testproblem-approx}) converges to $-0.1015$ when $M$ increases. 
This is because 
$\sigma_M(\cdot,s)$ 
converges to $\sigma(\cdot,s)$
as $M$ increases.

\begin{figure}[H]
  \centering
  \subfigure[]{
    \label{fig:Opt_portfolio_breakpoints}
    \includegraphics[scale=0.45]{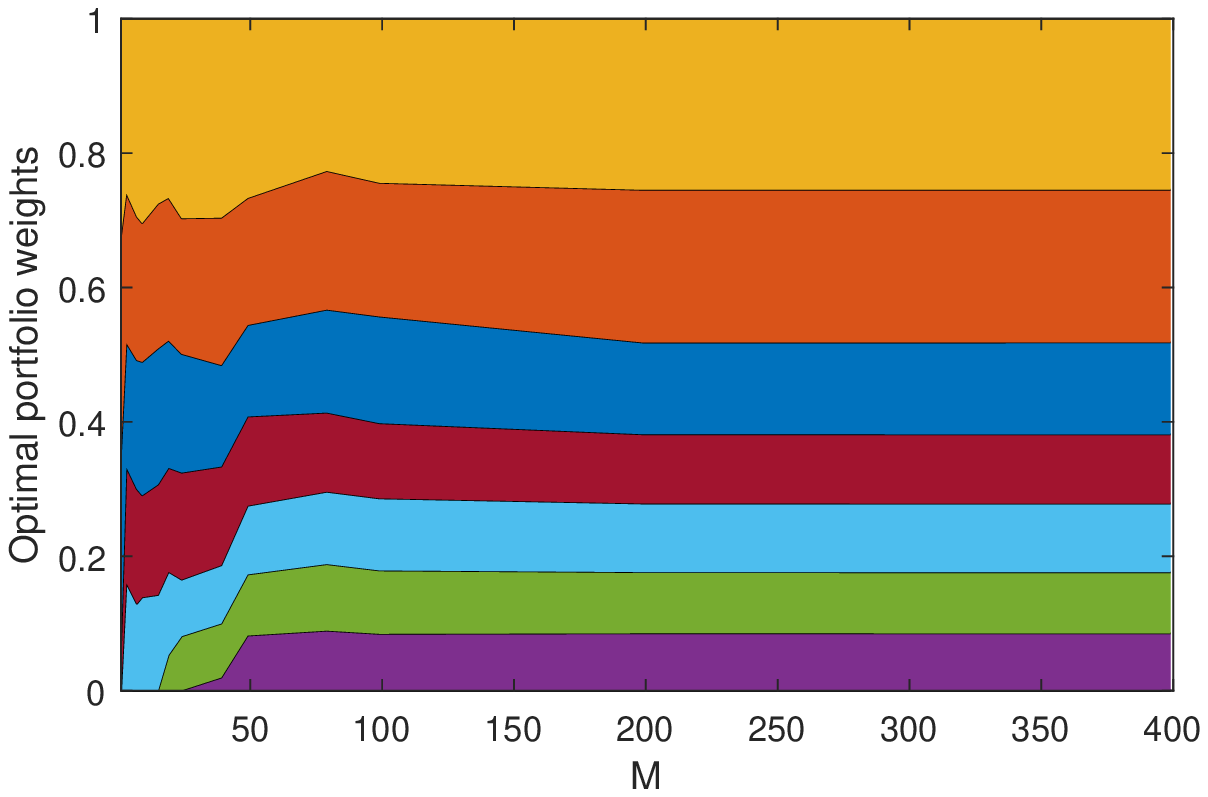}}
  \subfigure[]{
     \label{fig:Opt_value_breakpoints}
     \includegraphics[scale=0.45]{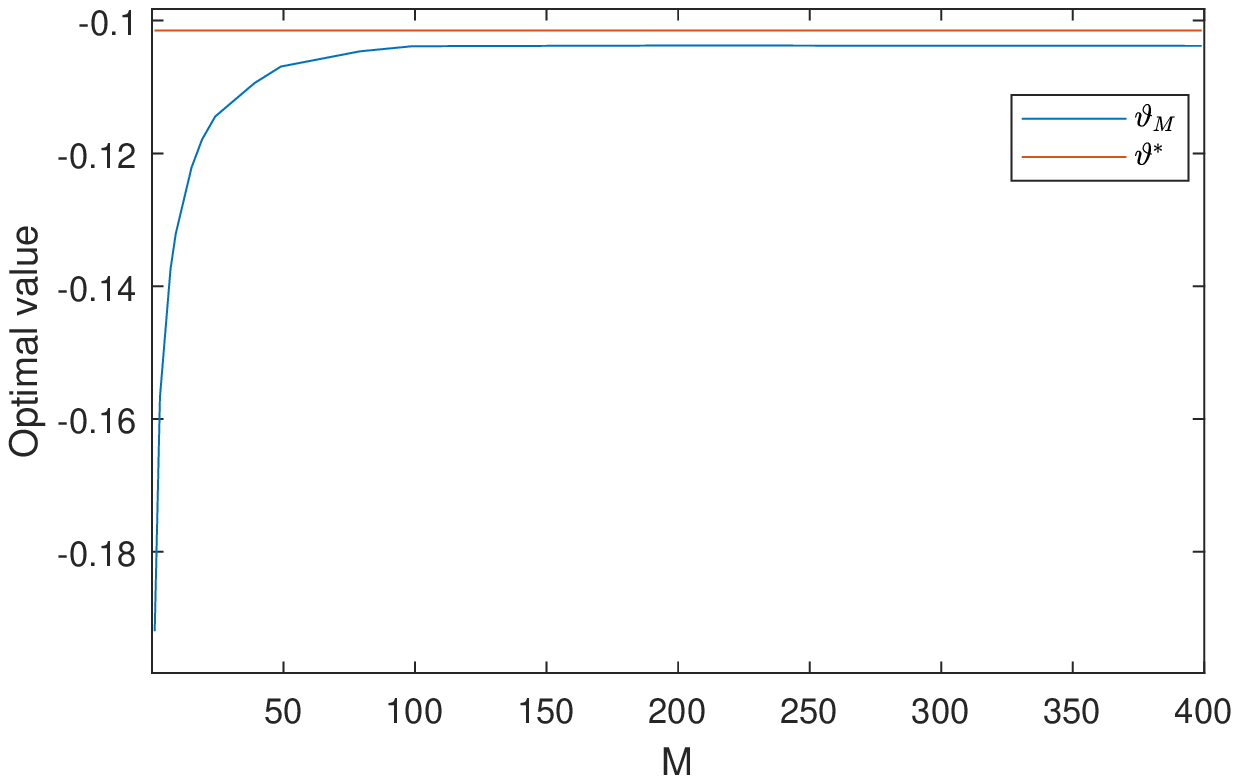}}
  \caption{{\color{black}(a) Optimal portfolio weights as $M$ increases over one simulation.
  (b) Optimal value as $M$ increases over one simulation.}}
  \label{fig:Opt_breakpoint}
\end{figure}
We have also  studied 
the impact of the sample size of $s$ on the optimal value and optimal portfolio weights. We solve problem (\ref{eq:testproblem-approx}) via Algorithm \ref{alm:DRSRM} using training datasets of cardinality $K=100$.
The radius of the ambiguity set 
is fixed with $r=0.1$,
and the number of breakpoints 
is also fixed at $M=99$.
The  sample size of $s$ varies with $N=10,50,100,200,300,500,1000$. 
Figures \ref{fig:Opt_portfolio_samples} and \ref{fig:Opt_value_boxplot} depict the convergence of the optimal portfolio weights $Z_N$ and the optimal value $\vartheta^N$ as the sample size $N$ increase. Figure \ref{fig:Opt_value_boxplot} implies that the  sequence of the optimal values $\vartheta^N$ converges.  

\begin{figure}[H]
  \centering
  \subfigure[]{
    \label{fig:Opt_portfolio_samples}
    \includegraphics[scale=0.45]{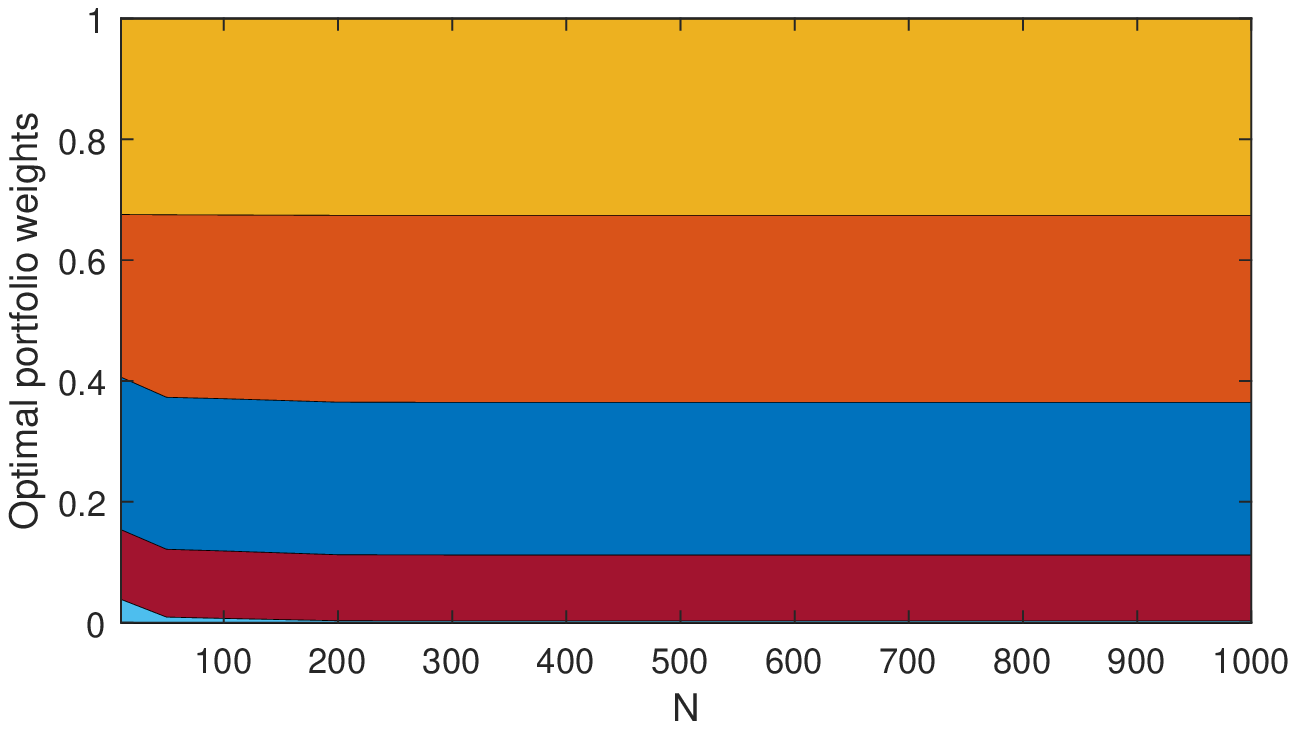}}
  \subfigure[]{
     \label{fig:Opt_value_boxplot}
     \includegraphics[scale=0.45]{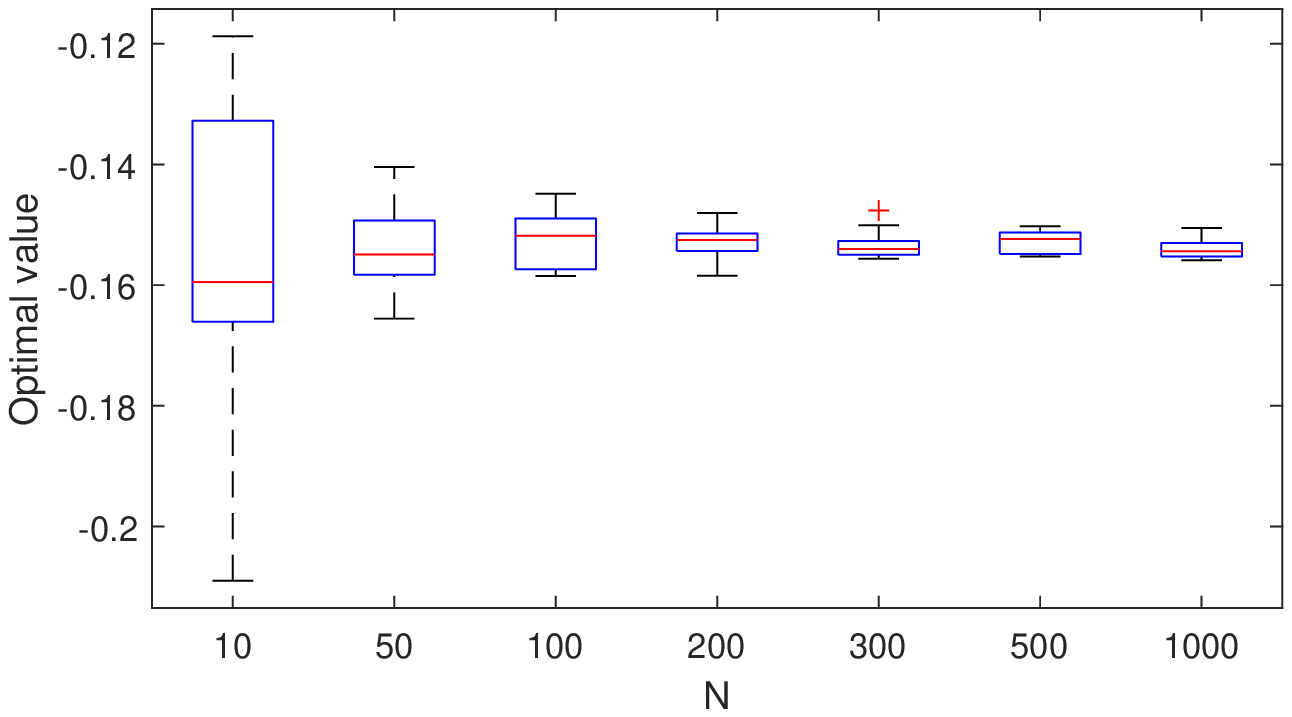}}
  \caption{{\color{black}(a) Optimal portfolio weights as $N$ increases averaged over 100 simulations.
  (b) Optimal value as $N$ increases over 100 simulations.}}
  \label{fig:two01}
\end{figure}
Finally, we investigate impact of the radius $r$ of the Kantorovich ball on the out-of-sample performance of model (\ref{eq:testproblem}).
Specifically, for each $r$,
we obtain 
an optimal solution $z_N(r)$  from solving problem (\ref{eq:testproblem}) and calculate  
check the change of 
ARSRM $\rho_{Q^*}(-z_N(r)^\top\xi)$ of  $-z_N(r)^\top\xi$
using the true distribution $Q^*$ and compare it with $\vt_N$.
Since $Q^*$ is unknown, we generate validation data of size 
$3000$ to evaluate $\rho_{Q^*}(-z_N(r)^\top\xi)$.
For each fixed $r$, we run 100 simulations with randomly generated training data of sample size $N=30$. 
 The blue shaded areas depict the tube of the optimal values $\vt_N$ between $20\%$ and $80\%$ quantiles. 
The solid black curve is
the sample mean. 
For each simulation, we examine whether inequality
$\rho_{Q^*}(-z_N(r)^\top\xi)\leq \vt_N$ holds or not,
and count once if it holds and zero otherwise.
The yellow depicts the percentages that the inequality 
holds out of 100 simulations when $r$ changes from $10^{-4}$ to $10^0$. According to \cite{ek18}, this is called reliability. To explain the idea, observe that
$\rho_{Q^*}(-z_N(r)^\top\xi)\geq \vt^*$ (where $\vt^*$ denotes the true ARSRM value)  and 
inequality $\rho_{Q^*}(-z_N(r)^\top\xi)\leq \vt_N$
implies that 
$\vt^*\leq \vt_N$. Thus, the more times the inequality holds, the more reliable that we may use 
$\vt_N$ as an upper bound of $\vt^*$.
The figure shows that when $r$ is increased to {\color{black} $10^{-1}$}, the reliability reaches $100\%$ which means that we do not need a large $r$ in this test.


\begin{figure}[H]
 \centering
\includegraphics[scale=0.5]{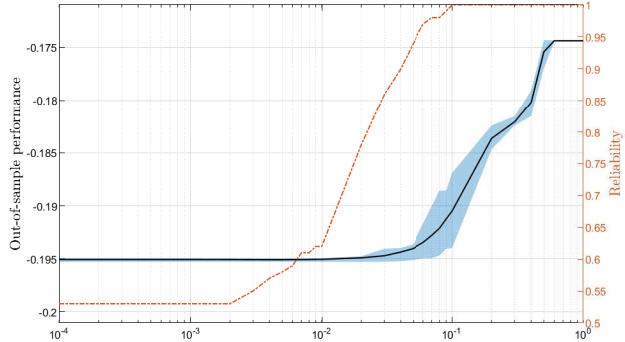}
\caption{Out of sample performance $\rho_{Q^*}(-z_N(r)^\top\xi)$ (left axis, solid line and shade area) and reliability
$\Prob(\rho_{Q^*}(-z_N(r)^\top\xi)\leq\vartheta_N)$ (right axis and dashed line) based on 100 independent simulations with $N = 30$
training samples.}
\label{fig:Out-of-sample}
\end{figure}

\subsection{Kusuoka's representation}

The second set of tests
concerns model (\ref{def:DR-RSRM-Opt}) with $\mathcal{M}$ being 
Kusouka's ambiguity set. In this case, the true 
risk spectrum takes a form of  
 $\sigma(t,\alpha)=\frac{1}{1-\alpha}\mathbbm{1}_{[\alpha,1]}(t)$ for $s\in (0,1)$.
To justify the application 
the model in this context,  we assume that
the portfolio manager's risk preference can be described by $\int_{0}^1\inmat{CVaR}_\alpha(\cdot)\mu(d\alpha)$ for some $\mu\in\mathscr{M}$
(see Assumption \ref{asm:Kasuoka-ambiguityset}). 
Consequently the minmax optimization model 
can be formulated as
\begin{equation*}
    \min_{z\in Z}\quad\sup_{\mu\in\mathcal{M}}\int_{0}^1\inmat{CVaR}_\alpha(-z^\top\xi)\mu(d\alpha).
\end{equation*}
Differing from the previous set of tests,
we 
use pairwise comparison approach to construct
the ambiguity set as outlined in
Section \ref{construction-Kasuoka's-set}.
Let $\mathcal{M}_{pair}$ denote the ambiguity set.
We consider the following program:
\begin{equation}
\label{eq:numerical_Kusuoka}
    \min_{z\in Z}\sup_{\mu\in\mathcal{M}_{pair}}\int_{0}^1\inmat{CVaR}_{\alpha}(-z^\top\xi)\mu(d\alpha),
\end{equation}
As discussed in Section \ref{subsec:Kusuoka's-representation}, we assume that $\alpha$ is discretely distributed  with finite support over $(0,1)$, i.e., $\mu(\alpha=\alpha_i)=m_i$ for $i\in[N]$. In addition, we assume that the true probability distribution of $\alpha$, denoted by $\mu^*$, satisfies $\mu^*(\alpha=\alpha_i)=m^*_i$ for $i\in[N]$.

\subsubsection{Design of Kusuoka's ambiguity set}
\label{subsec:designKusuoka-ambiguityset}

We  use
the {\em random
relative utility split scheme} (RRUS) 
considered by Armbruster and Delage \cite{ArD15}
to design the questions. 
Specifically, we ask the investor questions comparing a risky lottery  with two random outcomes and a
certain lottery with deterministic outcome, 
denoted respectively by
\begin{equation*}
Y_1 = \left\{\begin{aligned}
    y_1 \quad& \text{with \ probability $1-p$},\\
    y_3 \quad& \text{with \ probability \ $p$},
\end{aligned}\right.
\text{and} \ Y_2 = y_2.
\end{equation*}
Each question is described by four 
parameters $y_1<y_2<y_3$ and a probability $p$. Determine the
number of questions $J$. Below are the procedures.
\begin{itemize}
    \item [Step 0] Set $j=0$ and ${\cal M}^0:=\{\boldsymbol{m}\in\R_+^N:\sum_{i=1}^Nm_i=1\}$. For $j\leq J$, do Steps 1-3.
    \item [Step 1] Generate two random numbers denoted by $y_1^j$ and $y_3^j$ and assume for the convenience of exposition that $y_1^j<y_3^j$, let $p^j\in(0,1)$ be a positive number, which is randomly generated or designed. Define a lottery (a random variable $Y_1^j$) with 
    outcomes $y_1^j$ and $y_3^j$ 
    and respective probabilities $1 -p^j$ and $p^j$. 
    The risk value of $Y_1^j$ can be expressed as 
    \begin{equation*}
    \begin{split}
    \rho(Y_1^j)=&\sum_{i=1}^Nm_i\inmat{CVaR}_{\alpha_i}(Y_1^j)\\
    =&\sum_{i=1}^N m_i\left(\frac{y_1^j}{1-\alpha_i}\int_0^{1-p^j}\mathbbm{1}_{[\alpha_i,1]}(t)dt+\frac{y_3^j}{1-\alpha_i}\int_{1-p^j}^1\mathbbm{1}_{[\alpha_i,1]}(t)dt\right)\\
    =&\sum_{i=1}^N m_i\left(y_1^j\phi_i(p^j)+y_3^j\left(1-\phi_i(p^j)\right)\right),
    \end{split}
    \end{equation*}
    where $\phi_i(p)=\frac{1}{1-\alpha_i}\int_{0}^{1-p}\mathbbm{1}_{[\alpha_i,1]}(t)dt$ for $i\in[N]$.
    \item [Step 2] Calculate $$I_u:=\sup_{\boldsymbol{m}\in\mathcal{M}^j}\sum_{i=1}^N m_i\left(y_1^j\phi_i(p^j)+y_3^j\left(1-\phi_i(p^j)\right)\right)$$ and  $$I_l:=\inf_{\boldsymbol{m}\in\mathcal{M}^j}\sum_{i=1}^N m_i\left(y_1^j\phi_i(p^j)+y_3^j\left(1-\phi_i(p^j)\right)\right).$$
    Let $y_2^j=\frac{1}{2}(I_u+I_l)$.
    \item [Step 3] If $\sum_{i=1}^N m^*_i\left(y_1^j\phi_i(p^j)+y_3^j\left(1-\phi_i(p^j)\right)\right)\leq y_2^j$,
    then
    \begin{equation*}
        \mathcal{M}^{j+1}:=\mathcal{M}^{j}\cap\left\{\boldsymbol{m}\in\R_+^N:\sum_{i=1}^N m_i\left(y_1^j\phi_i(p^j)+y_3^j\left(1-\phi_i(p^j)\right)\right)\leq y_2^j\right\}.
    \end{equation*}
    Otherwise,
    \begin{equation*}
        \mathcal{M}^{j+1}:=\mathcal{M}^{j}\cap\left\{\boldsymbol{m}\in\R_+^N:\sum_{i=1}^N m_i\left(y_1^j\phi_i(p^j)+y_3^j\left(1-\phi_i(p^j)\right)\right)\geq y_2^j\right\}.  
    \end{equation*}
    Let $j:=j+1$. Go to Step 1.
 \end{itemize}
 In the $j$th iteration, Step 1 generates a lottery 
 with two 
 random outcomes $y_1^j$ and $y_3^j$ and 
 respective probabilities $1-p^j$ and $p^j$; 
 Step 2 provides an approach of choosing a certain outcome $y_2^j$, which can reduce the size of the ambiguity set efficiently when a new question is
added; Step 3 asks the DM to choose between the lottery and the
one with certain loss $y_2^j$. Here the true RM defined as $\rho^*(\cdot):=\sum_{i=1}^N m_i^*\inmat{CVaR}_{\alpha_i}(\cdot)$ is used to “act as the DM ”. After the DM makes a choice, a linear inequality is created and added to the ambiguity set
$\mathcal{M}^j$.
 
 \subsubsection{Tractable formulation of (\ref{eq:numerical_Kusuoka})}
 According to the 
 construction of the Kusuoka's ambiguity set
 outlined 
 in Section \ref{subsec:designKusuoka-ambiguityset}, 
 we can 
 write down 
 problem (\ref{eq:numerical_Kusuoka}) 
 as
 \begin{equation}
 \label{eq:num-kusuoka-tracfor}
     \begin{split}
         \min_{z\in Z}\quad&\sup_{\boldsymbol{m}\in\R_+^N} \sum_{i=1}^N m_i\inmat{CVaR}_{\alpha_i}(-z^\top\xi)\\
         \st\quad&\sum_{i=1}^N m_i = 1,\\
         &\sum_{i=1}^N m_i\left(y_1^j\phi_i(p^j)+y_3^j\left(1-\phi_i(p^j)\right)\right)\leq y_2^j, \ \forall\ j\in[J],
    \end{split}
 \end{equation}
and reformulate the latter via
 (\ref{def:CVaR-eq}) and Lagrange duality
as a single
linear programming problem
  \begin{alignat}{2}
  \label{eq:tractaleformulation-kusuoka}
          \min_{z\in Z,\eta,\boldsymbol{\lambda},\boldsymbol{\zeta}}\quad& \eta+\sum_{j=1}^J\lambda_jy_2^j \nonumber\\
          \st\quad&\boldsymbol{\lambda}\geq0,\\
          &\zeta_i+\frac{1}{1-\alpha_i}\mathbb{E}\left[(-z^\top\xi-\zeta_i)_+\right]\leq \eta+ \sum_{j=1}^J\lambda_j\left(y_1^j\phi_i(p^j)+y_3^j\left(1-\phi_i(p^j)\right)\right), \ \forall i\in[N],\nonumber
  \end{alignat}
where $\mathbb{E}[\cdot]$ denotes the expectation operator w.r.t.~the probability distribution of $\xi$.

 \subsubsection{Numerical results}
 In 
 this set of tests, 
 we pick up the elements of $\mathcal{A}=\{\alpha_1,\cdots,\alpha_N\}$ randomly over $(0,1)$, sort them in non-decreasing order, and 
 set $\mu^*(\cdot)=\frac{1}{N}\sum_{i=1}^N\delta_{\alpha_i}(\cdot)$. 
 By using 
 the relation
 $\sigma(t)=\int_{0}^t(1-\alpha)^{-1}\mu(d\alpha)$
 (see Section \ref{sec:introd}), and the right-continuity of risk spectrum $\sigma$, 
 we have
 \begin{equation}
 \label{relation-muandRS}
\sigma(t)=\sum_{i=0}^{N}\sum_{\ell=0}^{i}\frac{\mu(\alpha=\alpha_\ell)}{1-\alpha_\ell}\mathbbm{1}_{[\alpha_{i},\alpha_{i+1})}(t), 
 \end{equation}
  where $\alpha_0=0$ and $\alpha_{N+1}=1$.
  In this case,
  $\sigma(t)$ is a step-like function with breakpoints $\mathcal{A}$. 
  We derive the worst-case $\mu(\alpha=\alpha_i)$ for $i\in[N]$ by solving the optimization problem (\ref{eq:tractaleformulation-kusuoka}). 
  To 
  examine 
  the effectiveness 
 of the  pairwise comparison  
 approach outlined in 
 Section \ref{subsec:designKusuoka-ambiguityset},
 we use
 (\ref{relation-muandRS}) to visualize our 
 computational results
 by showing the corresponding step-like risk spectrum.
 
Figure \ref{fig:Numbet-OPT}
   depicts 
the change of
the optimal values of (\ref{eq:tractaleformulation-kusuoka}) 
as the number of questions increases.
The decresing 
trend is 
because the feasible set of inner maximization problem is reduced as the number of questionnaires increases. 
Figure \ref{fig:Numbet-PW} 
depicts 
the 
corresponding 
optimal portfolio weights 
as the number of questions 
increases from $5$ to $100$.


\begin{figure}[H]
  \centering
  \subfigure[]{
     \label{fig:Numbet-OPT}
     \includegraphics[scale=0.4]{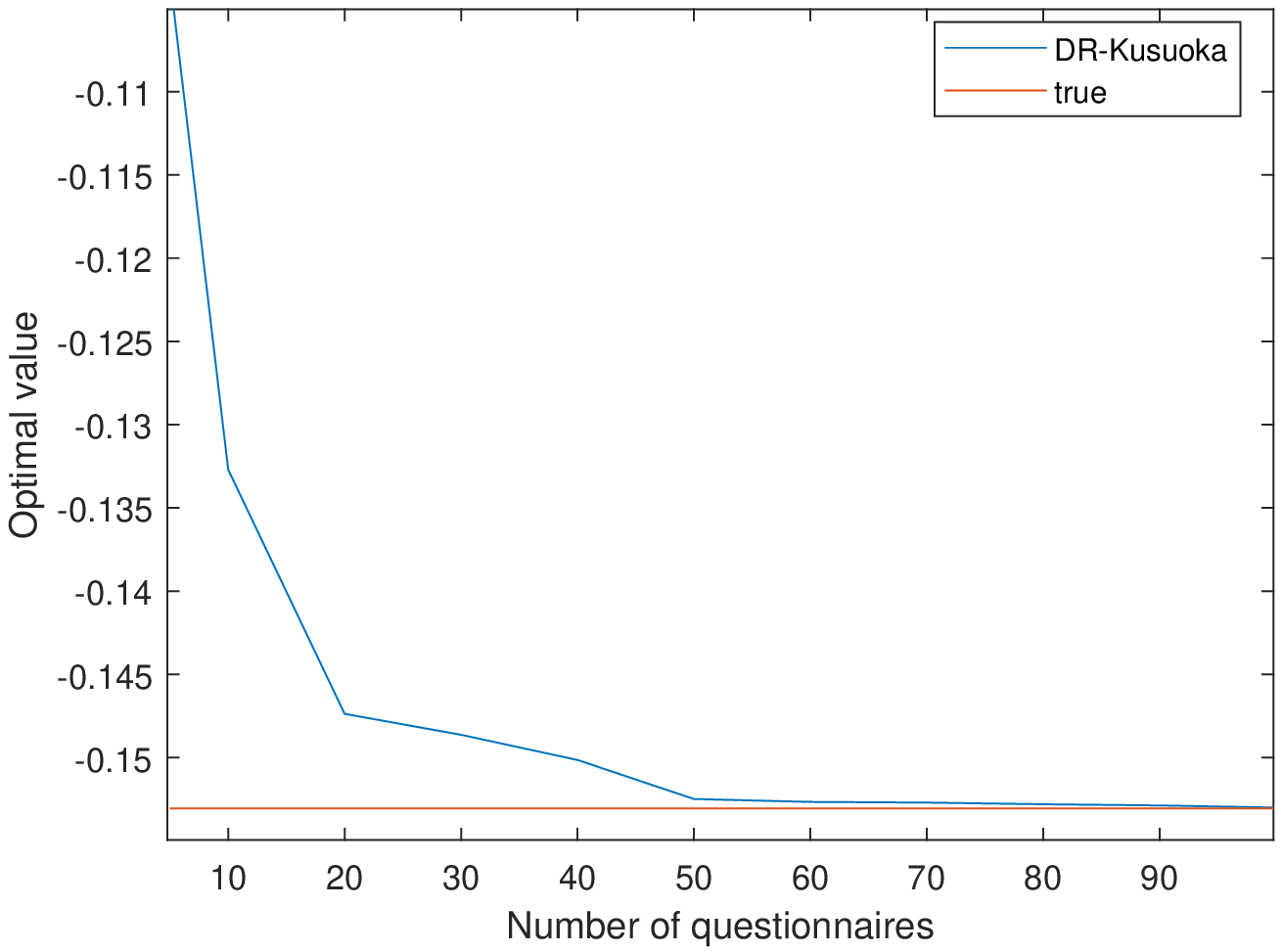}}
  \subfigure[]{
    \label{fig:Numbet-PW}
    \includegraphics[scale=0.4]{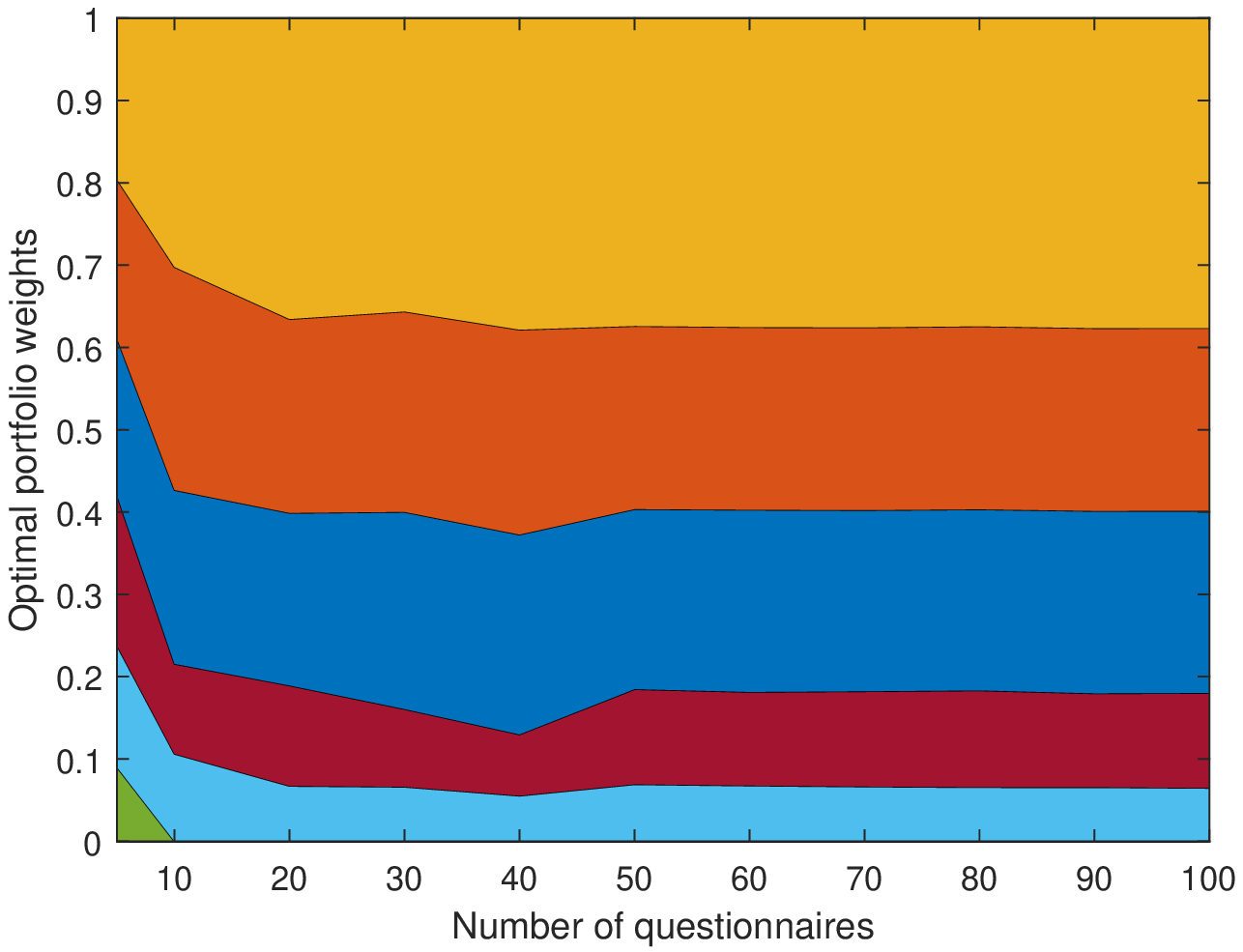}}
  \caption{The worst-case optimal value and portfolio weights after the number of questionnaires takes $J\in[5,10,20,30,40,50,60,70,80,90,100]$ in one simulation.}
  \label{fig:two}
\end{figure}
  Figure \ref{fig:True-PC-SLRS} depicts the changes of the worst-case risk spectrum constructed by solving problem (\ref{eq:tractaleformulation-kusuoka}) and using the relation (\ref{relation-muandRS}). 
  The result of Figure \ref{fig:True-PC-SLRS} shows that the worst-case step-like risk spectrum becomes more and more approximate to the true one as the number of questionnaire increases.

\begin{figure}[H]
 \centering
\includegraphics[scale=0.8]{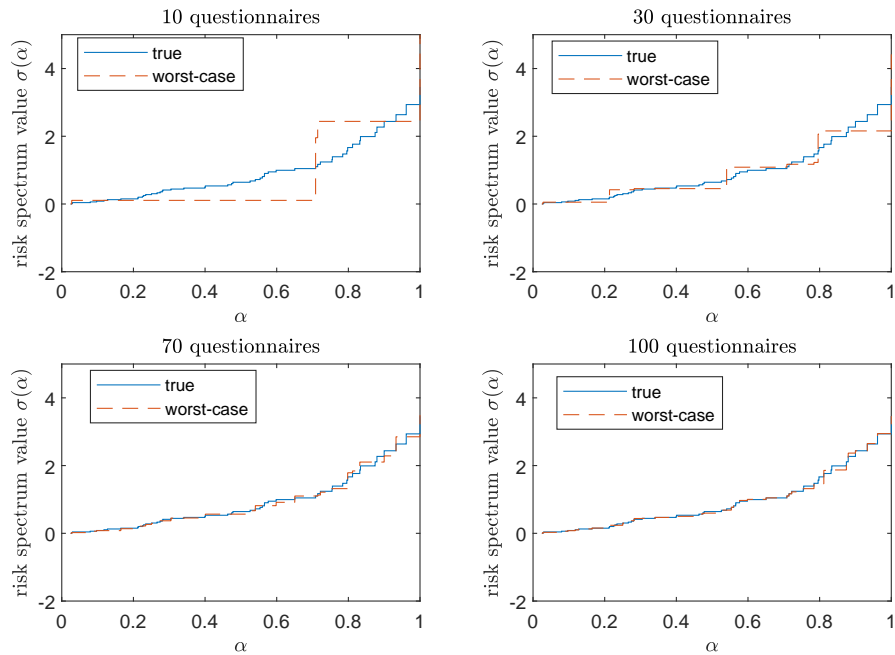}
\caption{The worst-case step-like risk spectrum after the number of questionnaires takes $J\in\{10,30,70,100\}$ in one simulation.}
\label{fig:True-PC-SLRS}
\end{figure}

\section{Extensions}

In the previous sections, we consider the case that the random risk spectra are step-like 
and the state variable is finitely distributed
to facilitate preference elicitation and numerical computation of the randomized SRM models. In practice, the risk spectrum of a DM's risk preference 
does not have to be step-like and also the state variable $s$ may be continuously distributed (which means the DM's risk preferences cannot be described by a finite number of SRM). 
If we apply the established models to the general cases,
where $\sigma(\cdot,s)$ is not step-like and $s$ is continuously distributed, then there will be inevitably modelling errors.  
In this section, we quantify the model errors so that we are guaranteed that the approximate models and computational   
schemes developed in the previous sections can be used within a prescribed precision.

\subsection{Static step-like approximation}

Let $Q^*$ be the true probability distribution of $s$. We consider the following expected SRM minimization problem:
\bgeqn\label{eq:ESRM-MIN}
\vartheta^* := \min_{z\in Z} \ \mathbb{E}_{Q^*}\left[\rho_{\sigma(\cdot,s)}\left(f(z,\xi)\right)\right],
\edeqn
where $\mathbb{E}_{Q^*}[\cdot]$ is the expectation value 
w.r.t.~$Q^*$, $\sigma(\cdot,s)\in\mathfrak{S}(S)\subseteq\mathscr{L}^q[0,1]$ and $q\in [1,\infty)$.
To ensure well-definiteness of (\ref{eq:ESRM-MIN}), we 
make the following assumption.

\begin{assumption}\label{asm:randomlossf}
Let $\mathscr{G}_f:=\{F_{f(z,\xi)}^\leftarrow(\cdot):z\in Z\}$. There exist a constant $p\in[1,\infty)$ with $\frac{1}{p}+\frac{1}{q}=1$ and a positive function $\psi_f\in\mathscr{L}^p[0,1]$ such that
$\sup_{g\in \mathscr{G}_f}|g(t)| \leq \psi_f(t)$
and $\mathbb{E}_Q\left[\int_0^1 \psi_f(t)\sigma(t,s)dt\right]$ $
<\infty$ for any $Q\in\mathscr{P}(S)$.
\end{assumption}
Without loss of generality, we assume that $\sigma(\cdot,s)$ is a general non-negative, non-decreasing function with the normalized property $\int_{0}^1\sigma(t,s)dt=1$ but it does not necessarily have a step-like structure. Let $\mathfrak{S}_M(S)
$ denote the set of all non-negative, non-decreasing, and normalized step-like functions over $[0,1]$ with breakpoints $\{t_1,\cdots,t_M\}$,
where $0<t_1<\cdots<t_M<1$ for all $s\in S$ and $S$ is a compact set. We 
consider 
{\color{black} the} step-like approximation of
$\sigma(\cdot,s)\in\mathfrak{S}(S)$.

\begin{definition}
Let $\sigma(\cdot,s)\in\mathfrak{S}(S)$. 
$\sigma_M(\cdot,s)$ is said to be a step-like approximation of $\sigma(\cdot,s)$ 
if $\sigma_M(t,s)=\sigma_i(s)$, for $t\in[t_i,t_{i+1})$, $i\in[M_0]$, such that $\sigma_i(s)\in[\sigma(t_i,s),\sigma(t_{i+1},s)]$, for $i\in[M_0]$, and $\int_{0}^1\sigma_M(t,s)=\sum_{i=0}^M\sigma_i(s)(t_{i+1}-t_i)=1$.
\end{definition}

We propose to obtain an approximate optimal value and optimal solution of problem (\ref{eq:ESRM-MIN}) by solving the following step-like approximate $\inmat{ARSRM}$ minimization problem:
\bgeqn\label{eq:Steplike-ESRM-MIN}
\vartheta_M^*:=\min_{z\in Z} \ \mathbb{E}_{Q^*}\left[\rho_{\sigma_M(\cdot,s)}(f(z,\xi))\right].
\edeqn
Since $\mathfrak{S}_M(S)\subset \mathfrak{S}(S)$, then 
$\mathbb{E}_{Q^*}\left[\rho_{\sigma_M(\cdot,s)}(f(z,\xi))\right]$ is well-defined under Assumption \ref{asm:randomlossf}.
Let $Z^*$ and $Z_M$ be the respective optimal solutions of problem (\ref{eq:ESRM-MIN}) and (\ref{eq:Steplike-ESRM-MIN}). Let
\begin{equation}
   \label{def:Delta}
\Delta_M := \max_{i\in[M_0]}\left|t_{i+1}-t_1\right|. 
\end{equation}
It is obvious that in order to ensure approximate validity, we need $\Delta_M$ to be sufficiently small, which is equivalent to setting $M$ a large value.

\begin{proposition}[Step-like approximation of random risk spectra]
\label{prop:steplikeapp-A}
Assume: (a)
for each fixed $s\in S$,
$\sigma(\cdot,s)$
is 
Lipschitz continuous over $[0,1]$
with modulus $L(s)$ 
where
$\mathbb{E}_{Q^*}[L(s)]<\infty$ and
the expectation is taken w.r.t.~the probability distribution of $s$,
(b) Assumption \ref{asm:randomlossf} holds and $\int_{0}^1\psi_f(t)dt<\infty$. 
Let $\sigma_M(\cdot,s)$ denote a step-like approximation of $\sigma(\cdot,s)$.
Then the following assertions hold.
\begin{itemize}
    \item [\inmat{($i$)}]
    Let 
    $\Delta_M$ be
    defined as  in (\ref{def:Delta}). Then
    \bgeqn
    |\vartheta^*-\vartheta_M^*|\leq \mathbb{E}_{Q^*}\left[L(s)\right]\Delta_M\int_{0}^1\psi_f(t)dt.
    \edeqn

    \item [\inmat{($ii$)}] Let $\{Z_M\}$ be a sequence of optimal solutions obtained form solving problem (\ref{eq:Steplike-ESRM-MIN}). 
    Then every cluster point of the sequence is an optimal solution of problem (\ref{eq:ESRM-MIN}), that is ,
    \bgeqn
    \lim_{M\rightarrow\infty}\mathbb{D}(Z_M,Z^*)=0,
    \edeqn
    where $\mathbb{D}(A,B)$ denotes the access distance of set $A$ over set $B$.
\end{itemize}
\end{proposition}
\begin{proof}
Part (i). Under Assumption \ref{asm:randomlossf}, $\rho_{\sigma(\cdot,s)}(f(z,\xi))$ is finite-valued
for each fixed $z\in Z$. Under conditions (a) and (b),
 \begin{equation*}\label{eq:lip}
 \begin{split}
     \left|\rho_{\sigma(\cdot,s)}(f(z,\xi))-\rho_{\sigma_M(\cdot,s)}(f(z,\xi))\right|
     &=\left|\sum_{i=0}^M\int_{t_i}^{t_{i+1}}F_{f(z,\xi)}^\leftarrow(\sigma(t,s)-\sigma_M(t,s))dt\right|\\
     &\leq\sum_{i=0}^M\int_{t_i}^{t_{i+1}}F_{f(z,\xi)}^\leftarrow|\sigma(t_{i+1},s)-\sigma(t_i,s)|dt\\
     &\leq\sum_{i=0}^ML(s)|t_{i+1}-t_i|\int_{t_i}^{t_{i+1}}\left|F_{f(z,\xi)}^\leftarrow(t)\right|dt\\
     &\leq L(s)\Delta_M\int_{0}^1\psi_f(t)dt<\infty,
 \end{split}
 \end{equation*}
 where the first inequality follows from the definition of $\sigma_M(t,s)$ and the fact that $\sigma(t,s)$ is non-decreasing w.r.t.~$t$. Moreover,
 under Assumption \ref{asm:randomlossf}, $\mathbb{E}_{Q^*}[\rho_{\sigma(\cdot,s)}(f(z,\xi))]$ is finite-valued.
 The discussions above imply that
\begin{equation*}
    \left|\mathbb{E}_{Q^*}[\rho_{\sigma(\cdot,s)}(f(z,\xi))]-\mathbb{E}_{Q^*}[\rho_{\sigma_M(\cdot,s)}(f(z,\xi))]\right|\leq  \mathbb{E}_{Q^*}[L(s)]\Delta_M\int_{0}^1\psi_f(t)dt.
\end{equation*}
Furthermore, since 
$Z$ is compact and $\rho_{\sigma(\cdot,s)}(f(z,\xi))$ is continuous in $z$ for all $\sigma(\cdot,s)\in\mathfrak{S}(S)$, then
\begin{equation*}
\begin{split}
|\vartheta^*-\vartheta_M^*|&\leq \sup_{z\in Z}\left|\mathbb{E}_{Q^*}[\rho_{\sigma(\cdot,s)}(f(z,\xi))]-\mathbb{E}_{Q^*}[\rho_{\sigma_M(\cdot,s)}(f(z,\xi))]\right|\\
&\leq  \mathbb{E}_{Q^*}[L(s)]\Delta_M\int_{0}^1\psi_f(t)dt.
\end{split}
\end{equation*}
Part (ii) follows directly from Part (i) and the well-known stability results in parametric programming \cite[Lemma 3.8]{lx13}. \hfill $\Box$
\end{proof}

Note that the conclusion of Proposition \ref{prop:steplikeapp-A} depends heavily on the assumption that $\sigma(\cdot,s)$ is Lipschitz continuous, which excludes many unbounded risk spectra such as Wang's
risk spectrum. In the case when $\sigma(\cdot,s)$ is not Lipschtz
continuous, one may set the breakpoints in a specific
way such that the following assumption is satisfied, we refer readers to \cite[Section 4.1]{wx20} for a
thorough discussion on this.


\subsection{Discretization of the state variable}

The computational schemes discussed in Section 6 rely  on the discrete distribution
of $s$. 
In practice,
the true probability distribution of $s$ may be unknown, but 
by can be estimated 
using 
sample data. Let $s^1,\cdots,s^N$ denote  i.i.d random sampling of $s$ and 
 $   Q_N=\frac{1}{N}\sum_{i=1}^N\delta_{s^i}(\cdot)
 $,
where $\delta_{s^i}(\cdot)$ denotes the Dirac probability measure at $s^i$.
We propose to approximate the true mean value of $\rho_{\sigma(\cdot,s)}(f(z,\xi))$, i.e., $\mathbb{E}_{Q^*}[\rho_{\sigma(\cdot,s)}(f(z,\xi))]$, using the sample average 
\begin{equation*}
    \mathbb{E}_{Q_N}\left[\rho_{\sigma_M(\cdot,s)}(f(z,\xi))\right]=\frac{1}{N}\sum_{i=1}^N\rho_{\sigma_M(\cdot,s^i)}(f(z,\xi)).
\end{equation*}
The next proposition gives a quantification of such an approximation.
\begin{proposition}[Sample average approximation of 
$s$]
\label{pro:7-2}
Assume  the settings and conditions of Proposition \ref{prop:steplikeapp-A}. For 
any $\varepsilon>0$ and $\theta>0$, let $N_0 :=-\frac{\ln\theta}{\Upsilon(\varepsilon)}$. Then
\bgeqn
\Prob\left(\min_{z\in Z}\left|\mathbb{E}_{Q_N}\left[\rho_{\sigma_M(\cdot,s)}(f(z,\xi))\right]-\mathbb{E}_{Q^*}\left[\rho_{\sigma(,s)}(f(z,\xi))\right]\right|\geq\varepsilon\right) \leq\theta
\edeqn
for all $N\geq N_0$ and $\Delta_M\leq\frac{\varepsilon}{2\mathbb{E}[L(s)]\int_0^1\psi_f(t)dt}$, where $\mathbb{E}[\cdot]$ is the expectation w.r.t.~the true distribution $Q^*$ of $s$.
\end{proposition}
\begin{proof}
By the triangle inequality, we have
\begin{equation*}
\begin{split}
  &\left|\mathbb{E}_{Q_N}[\rho_{\sigma_M(\cdot,s)}(f(z,\xi))]-\mathbb{E}_{Q^*}[\rho_{\sigma(\cdot,s)}(f(z,\xi))]\right|\\
 & \leq|\mathbb{E}_{Q_N}[\rho_{\sigma_M(\cdot,s)}(f(z,\xi))]-\mathbb{E}_{Q^*}[\rho_{\sigma_M(\cdot,s)}(f(z,\xi))]|\\
& \ \ \ +|\mathbb{E}_{Q^*}[\rho_{\sigma_M(\cdot,s)}(f(z,\xi))]-\mathbb{E}_{Q^*}[\rho_{\sigma(\cdot,s)}(f(z,\xi))].|  
\end{split}
\end{equation*}
By Proposition \ref{prop:steplikeapp-A}, we have
\begin{equation}
\label{eq:7.16}
\left|\mathbb{E}_{Q^*}[\rho_{\sigma(\cdot,s)}(f(z,\xi))]-\mathbb{E}_{Q^*}[\rho_{\sigma_M(\cdot,s)}(f(z,\xi))]\right|\leq  \mathbb{E}_{Q^*}[L(s)]\Delta_M\int_{0}^1\psi_f(t)dt.
\end{equation}
Let $\Delta_M$ be sufficiently small such that $\mathbb{E}_{Q^*}[L(s)]\Delta_M\int_{0}^1\psi_f(t)dt\leq\frac{\varepsilon}{2}$, i.e., $\Delta_M\leq\frac{\varepsilon}{2\mathbb{E}_{Q^*}[L(s)]\int_{0}^1\psi_f(t)dt}$.
By (\ref{eq:7.16})
\begin{equation}\label{eq:steplike-appC}
    \begin{split}
        &\min_{z\in Z}\left|\mathbb{E}_{Q_N}[\rho_{\sigma_M(\cdot,s)}(f(z,\xi))]-\mathbb{E}_{Q^*}[\rho_{\sigma(,s)}(f(z,\xi))]\right|\\
        \leq&\sup_{z\in Z}\left|\mathbb{E}_{Q_N}[\rho_{\sigma_M(\cdot,s)}(f(z,\xi))]-\mathbb{E}_{Q^*}[\rho_{\sigma_M(\cdot,s)}(f(z,\xi))]\right| \\
        & \ +\sup_{z\in Z}\left|\mathbb{E}_{Q^*}[\rho_{\sigma_M(\cdot,s)}(f(z,\xi))]-\mathbb{E}_{Q^*}[\rho_{\sigma(\cdot,s)}(f(z,\xi))]\right|\\
         \leq& \sup_{z\in Z}\left|\mathbb{E}_{Q_N}[\rho_{\sigma_M(\cdot,s)}(f(z,\xi))]-\mathbb{E}_{Q^*}[\rho_{\sigma_M(\cdot,s)}(f(z,\xi))]\right|+\frac{\varepsilon}{2}.
    \end{split}
\end{equation}
Moreover, 
\begin{equation}\label{eq:steplike-appD}
    \begin{split}
      &\left|\mathbb{E}_{Q_N}[\rho_{\sigma_M(\cdot,s}(f(z,\xi))]-\mathbb{E}_{Q^*}[\rho_{\sigma_M(\cdot,s)}(f(z,\xi))]\right|\\
      =&\left|\frac{1}{N}\sum_{j=1}^N\left[\sum_{i=0}^M\int_{t_i}^{t_{i+1}}\sigma_i(s^j)F_{f(z,\xi)}^\leftarrow(t)dt\right]-\mathbb{E}_{Q^*}\left[\sum_{i=0}^M\sigma_i(s)\int_{t_i}^{t_{i+1}}F_{f(z,\xi)}^\leftarrow(t)dt\right]\right|\\
      =&\left|\sum_{i=0}^M\left(\mathbb{E}_{Q_N}[\sigma_i(s)]-\mathbb{E}_{Q^*}[\sigma_i(s)]\right)\int_{t_i}^{t_{i+1}}F^\leftarrow_{f(z,\xi)}(t)dt\right|\\
      \leq&\max_{i\in[M_0]}\left|\mathbb{E}_{Q_N}[\sigma_i(s)]-\mathbb{E}_{Q^*}[\sigma_i(s)]\right|\int_{0}^1\left|F_{f(z,\xi)}^\leftarrow(t)\right|dt\\
      \leq&\|\mathbb{E}_{Q_N}\left[\boldsymbol{\sigma(s)}]-\mathbb{E}_{Q^*}[\boldsymbol{\sigma(s)}\right]\|_{\infty}\int_{0}^1\psi_f(t)dt,\\
    \end{split}
\end{equation}
where
$\boldsymbol{\sigma(s)}=(\sigma_0(s),\sigma_1(s),\cdots,\sigma_M(s))^\top$ and $\|\cdot\|_\infty$ denotes the infinity norm in $\R^{M+1}$.
Combining (\ref{eq:steplike-appC}) and (\ref{eq:steplike-appD}), we obtain
\begin{equation*}
    \begin{split}
     &\Prob\left(\min_{z\in Z}|\mathbb{E}_{Q_N}\rho_{\sigma_M(\cdot,s)}(f(z,\xi))-\mathbb{E}_{Q^*}[\rho_{\sigma(,s)}(f(z,\xi))]|\geq \varepsilon\right)\\
     \leq&\Prob\left(\|\mathbb{E}_{Q_N}[\boldsymbol{\sigma(s)}]-\mathbb{E}_{Q^*}[\boldsymbol{\sigma(s)}]\|_{\infty}\int_{0}^1\psi_f(t)dt\geq\frac{\varepsilon}{2}\right).
    \end{split}
\end{equation*}
Recall that the condition (b) of Proposition \ref{prop:steplikeapp-A}
ensures that $\int_{0}^1\psi_f(t)dt<\infty$. Thus by Cram$\acute{e}$r's large deviation theorem, there exists a positive integer $N_0$ and positive constant $\Upsilon(\varepsilon)$ (depending on $\varepsilon$ with $\Upsilon(0)=0$) such that for all $N\geq N_0$
\begin{equation*}
   \Prob\left(\|\mathbb{E}_{Q_N}[\boldsymbol{\sigma(s)}]-\mathbb{E}_{Q^*}[\boldsymbol{\sigma(s)}]\|_{\infty}\int_{0}^1\psi_f(t)dt\geq\frac{\varepsilon}{2}\right)\leq e^{-\Upsilon(\varepsilon)N}. 
\end{equation*}
For fixed $\theta\in(0,1)$, we let $N_0(\varepsilon,\theta):=-\frac{\ln\theta}{\Upsilon(\varepsilon)}$ and subsequently obtain
\begin{equation*}
   \Prob\left(\|\mathbb{E}_{Q_N}[\boldsymbol{\sigma(s)}]-\mathbb{E}_{Q^*}[\boldsymbol{\sigma(s)}]\|_{\infty}\int_{0}^1\psi_f(t)dt\geq\frac{\varepsilon}{2}\right)\leq \theta 
\end{equation*}
for all $N\geq N_0(\varepsilon,\theta)$. \hfill $\Box$
\end{proof}

\subsection{
Error bounds for the step-like approximation 
and discretization
}

Consider 
minimax optimization problem (\ref{def:DR-RSRM-Opt})
with the Kantorovich ambiguity set $\mathfrak{Q}^V_K(Q_V,r)$, which contains only 
discrete distributions: 
\begin{equation}
\label{eq:SLDR-ARSRM-K}
    \vartheta_{M}^{V}:=\min_{z\in Z}\sup_{Q\in\mathfrak{Q}^V_K(Q_V,r)}\mathbb{E}_Q[\rho_{\sigma_M(\cdot,s)}(f(z,\xi))],
\end{equation}
where 
\bgeqn
\mathfrak{Q}^V_K(Q_V,r):=\{Q\in\mathscr{P}(\mathcal{S}^V):\mathsf{dl}_K(Q,Q_V)\leq r\},
\edeqn
 $Q_V$ is 
defined as in (\ref{def:vt-projection}).
Let
$\mathfrak{Q}_K(Q_N,r)$
be the counterpart of 
$\mathfrak{Q}^V_K(Q_V,r)$
which contains both 
discrete and 
continuous distributions 
in the ball and
\begin{equation*}
    \vartheta_{M}^N:= \min_{z\in Z}\sup_{Q\in\mathfrak{Q}_K(Q_N,r)}\mathbb{E}_Q[\rho_{\sigma_M(\cdot,s)}(f(z,\xi))].
\end{equation*}
We want to quantify the difference between $\vt_M^{V}$ and
$\vartheta_{M}^N$.
Let 
$$
\mathfrak{F}:=\{h(\cdot):=\int_{0}^1F^\leftarrow_{f(z,\xi)}(t)\sigma_M(t,\cdot)dt, \ \forall z\in Z\},
$$
and  for any two probability measures $Q_1,Q_2\in\mathscr{P}(S)$, define pseudo-metric
\begin{equation*}
    \mathsf{dl}_\mathfrak{F}(Q_1,Q_2):=\sup_{z\in Z}\left|\mathbb{E}_{Q_1}\left[\int_{0}^1F^\leftarrow_{f(z,\xi)}(t)\sigma_M(t,s_1)dt\right]-\mathbb{E}_{Q_2}\left[\int_{0}^1F^\leftarrow_{f(z,\xi)}(t)\sigma_M(t,s_2)dt\right]\right|,
\end{equation*}
where $s_1$ and $s_2$ follow distributions $Q_1$ and $Q_2$ respectively.
For any two sets of probability measures $\mathfrak{Q}_1$ and $\mathfrak{Q}_2$, let
$ \mathbb{H}_\mathfrak{F}(\mathfrak{Q}_1,\mathfrak{Q}_2)$
denote the Hausdorff distance between the two sets under the pseudo-metric
$ \mathsf{dl}_\mathfrak{F}(Q_1,Q_2)$.


    Next, we quantify 
    the difference between $\vt^*$ and $\vt_M^N$.
    We need the following assumption.

\begin{assumption}\label{asm:steplike-RS}
    For
    each $\sigma(\cdot,s)\in\mathfrak{S}(S)$, its
     step-like approximation,
     $\sigma_M(\cdot,s)$,
     is Lipschitz continuous in $s$ over $S$,
     i.e.,
     for any fixed $t\in[0,1]$,
     there exists a 
     positive constant $L$ 
     such that $\|\sigma_{M}(t,s_1)-\sigma_M(t,s_2)\|\leq L\|s_1-s_2\|$ for all $s_1,s_2\in S$.   
\end{assumption}

To see how the assumption 
may be possibly satisfied, 
we take a look at the 
randomized risk spectra in  Example \ref{exm:rsrm}.
Consider $\sigma(t,\alpha)$ defined in (\ref{eq:step-RS-CVaR}). The function is step-like 
with breakpoint $\alpha$. In this case, any step-like approximation of the function is discontinuous in $\alpha$. Thus the Lipschitz condition fails to hold in this case.
Next, consider (\ref{eq:wang's}). The 
step-like approximation of $\sigma(\cdot,s)$ onto the space $\mathfrak{S}_M((0,1])$ can be written as:
     \begin{equation*}
         \sigma_M(t,s)=\left\{\begin{array}{lcl}
         \sigma(t_i,s), & & \inmat{for}\; t\in[t_i,t_{i+1}), i\in[M_0]\backslash\{M\},\\
          M+1-\sum_{i=0}^{M-1}\sigma(t_i,s),& & 
          \inmat{for}\;
          t\in[t_M,1),
         \end{array}\right.
     \end{equation*}
     where $\{t_1,\cdots,t_M\}=\{1/(M+1),\cdots,M/(M+1)\}$. It is easy to drive that
     for any $s_1,s_2\in (0,1]$ and $M\geq2$,
      \begin{equation*}
      |\sigma_M(t,s_1)-\sigma_M(t,s_2)| \leq\left\{\begin{array}{lcl}
         \left(M+1)(\ln(M+1)-1\right)|s_1-s_2|
         , & & \inmat{for}\; t\in[0,t_M), 
         \\
          M(M+1)(\ln(M+1)-1)|s_1-s_2|,
          & & 
          \inmat{for}\;
          t\in[t_M,1).
         \end{array}\right.
     \end{equation*}
     Finally, consider (\ref{eq:Gini}). 
    Since 
    $|\sigma(t,s_1)-\sigma(t,s_2)|\leq|s_1-s_2|$ for all $t\in[0,1]$, it is 
    easy to verify that its  step-like approximation 
    onto the space $\mathfrak{S}_M([0,1])$ is also Lipschtz continuous w.r.t.~$s$ with modulus $1$.


 We are now ready to state the main result of this section.

{\color{black}
\begin{theorem}
 Assume the settings and conditions of Proposition \ref{prop:steplikeapp-A}. 
 Assume, in addition, that (a) Assumption \ref{asm:steplike-RS} holds, 
 (b) there exist positive constants $C$, $v$, and $\delta_0$ such that $Q^*(|s-s_0|\leq\delta)>C\delta^{v}$ for any fixed point $s_0\in S$ and $\delta\in(0,\delta_0)$, where $Q^*$ is a true continuous distribution of $s$. 
 Then
 for any $\theta\in(0,1)$  and $r=r_N(\theta)$, there exist positive constants $V_0$ and $N_0$ such that  
\bgeqn
\vartheta^* \in [\vartheta_{M}^V -3L \Psi r_N(\theta), \vartheta_{M}^V+3L \Psi r_N(\theta)]
\edeqn
with probability at least $1-\theta$ for all $V\geq V_0$, $N\geq N_0$, 
where $\Psi =\int_{0}^1\psi_f(t)dt$, $r_N(\theta)$ is defined as in 
(\ref{eq:rN}) and $r_N(\theta)\leq 1$,
$$
V_0 = -\frac{\ln\left(\theta/(2\varrho_1(3L \Psi r_N(\theta)))\right)}{\varrho_2(3L \Psi r_N(\theta))}, \  N_0 =\max\left\{\frac{\log(C_1\theta^{-1})}{C_2  r_N^2(\theta)}, -\frac{\ln(\theta/2)}{\Upsilon(3L \Psi r_N(\theta)}\right\}
$$
and 
$\Delta_M\leq\frac{3L r_N(\theta)}{2\mathbb{E}[L(s)]
}$.


\end{theorem}

\begin{proof}
Under Assumption \ref{asm:steplike-RS}, we can derive
the Lipschitz continuity of
$\rho_{\sigma_M(\cdot,s)}(f(z,\xi))$
in $s$, that is,
\begin{equation*}
\begin{split}
\left|\rho_{\sigma_M(\cdot,s_1)}(f(z,\xi))-\rho_{\sigma_M(\cdot,s_2)}(f(z,\xi))\right|
=&\left|\sum_{i=0}^M(\sigma_i(s_1)-\sigma_i(s_2))\int_{t_i}^{t_{i+1}}F^\leftarrow_{f(z,\xi)}(t)dt\right|\\
\leq&\max_{i\in[M_0]}|\sigma_{i}(s_1)-\sigma_i(s_2)|\int_{0}^1\psi_f(t)dt
\leq L\Psi|s_1-s_2| 
\end{split}
\end{equation*}
where 
$\Psi=\int_{0}^1\psi_f(t)dt$.
By the definition of the Kantorovich metric, the Lipschitz continuity ensures
that  for any $Q\in\mathscr{P}(S)$
\bgeqn
    \left|\mathbb{E}_Q[\rho_{\sigma_M(\cdot,s)}(f(z,\xi))]-\mathbb{E}_{Q_N}[\rho_{\sigma_M(\cdot,s)}\left(f(z,\xi)\right)]\right|\leq L\Psi\mathsf{dl}_K(Q,Q_N).
\label{eq:Q-QN-thm7.1}
\edeqn
By virtue of \cite[Theorem 3]{PichlerXu18}
and \cite[Theorem 2]{ChenSunXu21},
\begin{equation*}
\begin{split}
    \left|\vartheta_M^{V}-\vartheta_{M}^N\right|&\leq \mathbb{H}_\mathfrak{F}\left(\mathfrak{Q}_K\left(Q_N,r\right),\mathfrak{Q}_K^V\left(Q_V,r\right)\right)\\
    &\leq L\Psi \mathbb{H}_K\left(\mathfrak{Q}_K\left(Q_N,r\right),\mathfrak{Q}_K^V\left(Q_V,r\right)\right)
    \leq 3L\Psi\beta_V,
\end{split}
\end{equation*}
where 
$
\beta_V = \max_{s\in S}\min_{1\leq j \leq V} d(s,\hat{s}^j),
\label{def:beta_V}
$
and $d(s,\hat{s}^j)=\|s-\hat{s}^j\|$ denotes some norm distance in the Euclidean space. Note that $\beta_V$ depends on the selection of $\{\hat{s}^1,\cdots,\hat{s}^V\}$, 
thus $\beta_V$
is a random variable.
By the triangle inequality
\begin{equation}\label{eq:prob-pro1}
    \begin{split}
        \Prob(|\vartheta_{M}^V-\vartheta^*|\geq\varepsilon)
        \leq&\Prob\left(\left(\left|\vartheta_M^V-\vartheta_{M}^N\right|+\left|\vartheta_{M}^N-\vartheta^*\right|\right)\geq\varepsilon\right)\\
         \leq&\Prob\left(3L\Psi\beta_V\geq\frac{\varepsilon}{3}\right)+\Prob\left(\left|\vartheta_{M}^N-\vartheta^*\right|\geq\frac{2\varepsilon}{3}\right)\\
        \leq&\Prob\left(\beta_V\geq\frac{\varepsilon}{9L\Psi}\right)+\Prob\left(\left|\vartheta_{M}^N-\vartheta^*\right|\geq\frac{2\varepsilon}{3}\right).
\end{split}
\end{equation}
for any small real number $\varepsilon>0$. 
In the following, we show that there exist constants $V_0$, $N_0$ and $M_0$ such that $\Prob\left(\beta_V\geq\frac{\varepsilon}{9L\Psi}\right)\leq\frac{\theta}{2}$ and $\Prob\left(\left|\vartheta_{M}^N-\vartheta^*\right|\geq\frac{2\varepsilon}{3}\right)\leq\frac{\theta}{2}$ for all $V\geq V_0$, $N\geq N_0$ and $M\geq M_0$.
We proceed the rest of the proof 
in two steps.

{\bf Step 1.} Since the support set $S$ of $s$ is bounded and the condition (b) holds, then
it follows by \cite{AXZ14}, \cite[Lemma 3.1]{XLS18} and \cite[Proposition 9]{lpx19} that
there exist positive constants $\varrho_1(\varepsilon)$ and $\varrho_2(\varepsilon)$ depending on $\varepsilon$ such that
$$
\Prob\left(\beta_V\geq\frac{\varepsilon}{9L\Psi}\right)\leq \varrho_1(\varepsilon) e^{-\varrho_2(\varepsilon)N}. 
$$
Let
$V_0 :=-\frac{\ln\left(\theta/(2\varrho_1(\varepsilon))\right)}{\varrho_2(\varepsilon)}$ be 
such that $\Prob\left(\beta_V\geq\frac{\varepsilon}{9L\Psi}\right)\leq\frac{\theta}{2}$ for all $V\geq V_0$.

{\bf Step 2.} For each fixed $z\in Z$ and $Q\in\mathfrak{Q}_K(Q_N,r)$,
\begin{equation*}
    \begin{split}
        \left|\mathbb{E}_Q[\rho_{\sigma_M(\cdot,s)}(f(z,\xi))]-\mathbb{E}_{Q^*}[\rho_{\sigma(\cdot,s)}(f(z,\xi))]\right|
        \leq&\left|\mathbb{E}_Q[\rho_{\sigma_M(\cdot,s)}(f(z,\xi))]-\mathbb{E}_{Q_N}[\rho_{\sigma_M(\cdot,s)}(f(z,\xi))]\right|\\
        &+\left|\mathbb{E}_{Q_N}[\rho_{\sigma_M(\cdot,s)}(f(z,\xi))]-\mathbb{E}_{Q^*}[\rho_{\sigma(\cdot,s)}(f(z,\xi))]\right|.
    \end{split}
\end{equation*}
Let $E$ denote the event that $\mathsf{dl}_K(Q^*,Q_N)\leq r_N(\theta)$ for $r_N(\theta)\leq 1$ and $F$ denote the event that 
\bgeqn 
\left|\mathbb{E}_{Q_N}[\rho_{\sigma_M(\cdot,s)}(f(z,\xi))]-\mathbb{E}_{Q^*}[\rho_{\sigma(\cdot,s)}(f(z,\xi))]\right|\leq\frac{\varepsilon}{3}.
\label{eq:Q_N-Q*-step2}
\edeqn
By (\ref{pro:wass}), (\ref{eq:rN}) and Proposition \ref{pro:7-2}, there exists a positive integer 
 $
 N_1
 :=\max\left\{\frac{\log(2C_1\theta^{-1})}{C_2},-\frac{\ln(\theta/2)}{\Upsilon(\varepsilon)}\right\}
 $
 and $M_0$
 such that $\Prob(E\cap F)\geq 1-\frac{\theta}{2}$ and $\Delta_M\leq\frac{\varepsilon}{2\mathbb{E}[L(s)]\Psi}$ for all $N\geq N_2$ and $M\geq M_0$.
  Consequently
    \bgeqn
        &&\Prob\left(\left|\vartheta_{M}^N-\vartheta^*\right|\geq\frac{2\varepsilon}{3}\right)\nonumber\\
        &\leq&\Prob\left(\sup_{z\in Z,Q\in\mathfrak{Q}_K(Q_N,r)}\left|\mathbb{E}_Q[\rho_{\sigma_M(\cdot,s)}(f(z,\xi))]-\mathbb{E}_{Q^*}[\rho_{\sigma(\cdot,s)}(f(z,\xi))]\right|\geq\frac{2\varepsilon}{3}\right)\nonumber\\
        &=&\Prob\left(\sup_{z\in Z,Q\in\mathfrak{Q}_K(Q_N,r)}\left|\mathbb{E}_Q[\rho_{\sigma_M(\cdot,s)}(f(z,\xi))]-\mathbb{E}_{Q^*}[\rho_{\sigma(\cdot,s)}(f(z,\xi))]\right|\geq\frac{2\varepsilon}{3}\bigcap(E\cap F)\right)\nonumber\\
        &&+\Prob\left(\sup_{z\in Z,Q\in\mathfrak{Q}_K(Q_N,r)}\left|\mathbb{E}_Q[\rho_{\sigma_M(\cdot,s)}(f(z,\xi))]-\mathbb{E}_{Q^*}[\rho_{\sigma(\cdot,s)}(f(z,\xi))]\right|\geq\frac{2\varepsilon}{3}\bigcap(\overline{E\cap F})\right)\nonumber\\
        &\leq&\Prob\left(\sup_{z\in Z,Q\in\mathfrak{Q}_K(Q_N,r)}\left(\left|\mathbb{E}_Q[\rho_{\sigma_M(\cdot,s)}(f(z,\xi))]-\mathbb{E}_{Q_N}[\rho_{\sigma_M(\cdot,s)}(f(z,\xi))]\right|\right.\right.\nonumber\\
        &&\left.\left.+\left|\mathbb{E}_{Q_N}[\rho_{\sigma_M(\cdot,s)}(f(z,\xi))]-\mathbb{E}_{Q^*}[\rho_{\sigma(\cdot,s)}(f(z,\xi))]\right|\right)\geq\frac{2\varepsilon}{3}|E\cap F\right)\Prob(E\cap F)+\Prob(\overline{E\cap F})\nonumber\\
        &\leq&\Prob\left(\sup_{z\in Z,Q\in\mathfrak{Q}_K(Q_N,r_N(\theta))}
        \left|\mathbb{E}_Q [\rho_{\sigma_M(\cdot,s)}(f(z,\xi))]-\mathbb{E}_{Q_N}[\rho_{\sigma_M(\cdot,s)}(f(z,\xi))]\right|\geq\frac{\varepsilon}{3}\right)\Prob(E\cap F)\nonumber\\
        &&\quad\quad+
        \frac{\theta}{2}\nonumber\\
        &\leq&\Prob\left(\sup_{Q\in\mathfrak{Q}_K(Q_N,r_N(\theta))}L\Psi\mathsf{dl}_K(Q,Q_N)\geq\frac{\varepsilon}{3}\right)+\frac{\theta}{2},
\label{eq:prob-pro}
\edeqn
where the second last inequality results from 
(\ref{eq:Q_N-Q*-step2}) and the last inequality 
is due to (\ref{eq:Q-QN-thm7.1}).
Let $N_2 :=\frac{9L^2\Psi^2\log(C_1\theta^{-1})}{C_2\varepsilon^2}$. Then
$r_N(\theta)<\frac{\varepsilon}{3L\Psi}$, 
i.e., $\varepsilon> 3 L \Psi r_N(\theta)$ for all $N\geq N_2$.
Summarizing the discussions above, 
we conclude that 
there exist $V_0:=\frac{\ln\left(\theta/(2\varrho_1(\varepsilon))\right)}{\varrho_2(\varepsilon)}$, $N_0 :=\max\left\{N_1,N_2\right\}$ and $M_0$ such that
\begin{equation*}
    \begin{split}
        \Prob(|\vartheta_{M}^V-\vartheta^*|\geq\varepsilon)\leq
        &\Prob\left(\beta_V\geq\frac{\varepsilon}{9L\Psi}\right)+\Prob\left(\sup_{Q\in\mathfrak{Q}_K(Q_N,r_N(\theta))}\mathsf{dl}_K(Q,Q_N)> r_N(\theta)\right)+\frac{\theta}{2}\\
        \leq&\frac{\theta}{2}+\frac{\theta}{2}=\theta
    \end{split}
\end{equation*}
for all $V\geq V_0$, $N\geq N_0$ and $M\geq M_0$. The inequality implies that
$$
\vartheta^*\in[\vartheta_{M}^V-3 L\Psi r_N(\theta),\vartheta_{M}^V+3 L\Psi r_N(\theta)]
$$
with probability at least $1-\theta$ for all $V\geq V_0$, $N\geq N_0$ and $M\geq M_0$.
\hfill $\Box$
\end{proof}
}

\section{Concluding remarks}

In this paper, we explore randomization of 
spectral risk measure for the case that
a single SRM does not exist for the description of 
a DM's risk preferences. 
Differing from
Bertsimas and O'Hair's method  \cite{BeO13} for 
tackling DM's preference inconsistency 
where the authors regard  inconsistencies 
as ``mistakes'' and consequently propose 
a remedy by relaxing the model 
to accommodate the mistakes so long as the total quantity of mistakes is controlled, 
we allow unlimited 
number of inconsistencies/mistakes. 
As such, the proposed model may be more easily
utilized for descriptive analysis where the empirical data used to describe a DM's past behaviour.
It can perhaps also be more realistically 
used for prescriptive analysis by a modeller who does not have complete information on the DM's  
risk preference and consequently uses the available data to forecast the DM's future decisions. 
Moreover, 
the randomization of risk measures enables us to 
interpret the Kusuoka's representation and spectral risk representation (\ref{eq:RSRM}) of a law invariant risk measure from risk preference perspective and  
provide an avenue to construct an approximation of the ambiguity sets in these representations via preference elicitation.
As we can see, the randomization 
depends on the ``building blocks'', e.g., 
VaRs, or CVaRs or SRMs. 
We envisage that this kind of randomization approaches can be extended to other risk measures. Moreover, 
it will be interesting to investigate how to ``learn''
efficiently the DM's preferences
in terms of the type of $\sigma(\cdot,s)$ and the distribution of $s$ in practical applications.
We leave all these for future research.


\appendix
\section{Supplementary materials}
\label{Append}

\subsection{Proof of Proposition \ref{pro:algthm-convergence}}
\label{pro:algthm-convergence-proof}
\begin{proof}
The proof is analogous to 
that of \cite[Proposition 3.1 ]{wx20}. Define
\bgeqn\label{alm:min-maxproblem-1}
 v(z,\boldsymbol{q}) := \sum_{i=1}^V\sum_{k=1}^Kq_i\beta_{ik}\inmat{CVaR}_{\alpha_k}(f(z,\xi)),
\edeqn
 where $\beta_{ik}=(\psi_{i,k}-\psi_{i,k-1})(K-k+1)$, $\alpha_k=\frac{k-1}{K}$ ,$\psi_{i,0}=0$ and $\psi_{i,k}=\int_{\frac{k-1}{K}}^{\frac{k}{K}}\sigma_M(t,\hat{s}^i)dt=\sum_{j=0}^M\sigma_j(\hat{s}^i)\int_{\frac{k-1}{K}}^{\frac{k}{K}}\mathbbm{1}_{[t_i,t_{i+1})}(t)dt$ for all $i\in[V],k\in[K]$. It is 
 easy to observe
 that $v(z,\boldsymbol{q})$ is  convex 
 in $z$ for 
 fixed $\boldsymbol{q}$, and $v(z,\boldsymbol{q})$ is 
 linear 
 $\boldsymbol{q}$ for every fixed $z$. Thus, we can rewrite problem (\ref{def:DR-RSRM-Opt-approx}) 
 as
 \bgeqn
 \label{alm:min-maxproblem}
 \min_{z\in Z} \max_{\boldsymbol{q}}  v(z,\boldsymbol{q}).
 \edeqn
Let
  $(z^*,\boldsymbol{q}^*)$ denote a cluster point of $\{(z^{\ell+1},\boldsymbol{q}^{\ell+1})\}$ generated by Algorithm \ref{alm:DRSRM}. 
We want to show that
\begin{equation}
\label{eq:saddle-point}
    v(z^*,\boldsymbol{q})\leq v(z^*,\boldsymbol{q}^*)\leq v(z,\boldsymbol{q}^*), \ \forall \ \boldsymbol{q}\in\mathfrak{Q}^V, z\in Z,
\end{equation}
which means that $(z^*,\boldsymbol{q}^*)$ is a saddle point of $v(z,\boldsymbol{q})$ and hence an optimal solution of (\ref{def:DR-RSRM-Opt-approx}). 
For $\ell=0,1,2,\cdots$, it follows by the algorithm
\begin{equation}
    v(z^{\ell},\boldsymbol{q})\leq v(z^{\ell},\boldsymbol{q}^{\ell+1}), \ \forall \boldsymbol{q}\in\mathfrak{Q}^V
\end{equation}
and
\begin{equation}\label{ieq:contra}
    v(z^{\ell+1},\boldsymbol{q}^{\ell+1})\leq v(z,\boldsymbol{q}^{\ell+1}), \ \forall z\in Z.
\end{equation}
In the case when the algorithm terminates in finite steps,  $z^{\ell+1}=z^\ell$ and $\boldsymbol{q}^{\ell+1}=\boldsymbol{q}^\ell$ for some $\ell$ and thus $(z^{\ell+1},\boldsymbol{q}^{\ell+1})$ satisfies (\ref{eq:saddle-point}).

Next, we consider the case where the algorithm generates an infinite sequence $\{(z^{\ell+1},\boldsymbol{q}^{\ell+1})\}$. Let $(\tilde{z},\tilde{\boldsymbol{q}})$ be a cluster point of $\{(z^{\ell+1},\boldsymbol{q}^{\ell+1})\}$,i.e., $(z^{\ell+1},\boldsymbol{q}^{\ell+1})\rightarrow(\tilde{z},\tilde{\boldsymbol{q}})$ as $l\rightarrow\infty$. Assume for a contradiction that $(\tilde{z},\tilde{\boldsymbol{q}})$ is not a solution to 
problem (\ref{alm:min-maxproblem}). Then $(\tilde{z},\tilde{\boldsymbol{q}})$ 
does not satisfy one of the inequalities in (\ref{eq:saddle-point}). Let's consider the case that the second inequality 
fails to hold. Then there exists $z_0\in Z$ such that
\bgeqn
v(\tilde{z},\tilde{\boldsymbol{q}})>v(z_0,\tilde{\boldsymbol{q}}),
\edeqn
and subsequently, we have 
\bgeqn
v(z^{\ell+1},\boldsymbol{q}^{\ell+1})> v(z_0,\tilde{q})
\edeqn
for $\ell$ sufficiently large, which is 
contradiction 
(\ref{ieq:contra}) as desired. 
Likewise,
we can 
show that $(\tilde{z},\tilde{\boldsymbol{q}})$ satisfies the first inequality in (\ref{eq:saddle-point}).  \hfill $\Box$
\end{proof}

\end{document}